\documentclass[11pt]{amsart}
\usepackage{amsmath,amsfonts,amssymb,amsthm,epsfig,color,tikz}
\newcommand{\bb}{\mathbb}
\newcommand{\p}{\overrightarrow{p}}
\newcommand{\x}{\overrightarrow{x}}
\newcommand{\C}{\bb C}
\newcommand{\h}{\bb H}
\newcommand{\Z}{\bb Z}
\newcommand{\R}{\bb R}
\newcommand{\N}{\bb N}
\newcommand{\Q}{\bb Q}

\newtheorem{obs}{Observation}
\newcommand{\La}{\Lambda}

\DeclareMathOperator{\area}{area}
\newcommand{\eps}{\varepsilon}

\newcommand{\Om}{\Omega}

\newcommand{\B}{\mathcal B}
\newcommand{\F}{\mathcal F}

\newcommand{\vv}{\mathbf{v}}

\newtheorem{Theorem}{Theorem}
\numberwithin{Theorem}{section}
\newtheorem{Cor}[Theorem]{Corollary}

\newtheorem{Prop}[Theorem]{Proposition}
\newtheorem{lemma}[Theorem]{Lemma}

\newtheorem*{lemma*}{Lemma}
\newtheorem*{question*}{Question}
\newtheorem*{conj}{Conjecture}

\newtheorem*{theorem*}{Theorem}

\numberwithin{equation}{section}
\begin{document}
\title[A Poincar\'e section for horocycle flow]{A Poincar\'e section for horocycle flow on the space of lattices}
\author{J.~S.~Athreya}
\author{Y.~Cheung}
\subjclass[2000]{primary: 37A17; secondary 37-06, 37-02}
\email{jathreya@illinois.edu}
\email{ycheung@sfsu.edu}
\address{Department of Mathematics, University of Illinois Urbana-Champaign, 1409 W. Green Street, Urbana, IL 61801}
\address{Department of Mathematics, San Francisco State University, Thornton Hall 937,
1600 Holloway Ave, San Francisco, CA 94132}
\begin{abstract}
We construct a Poincar\'e section for the horocycle flow on the modular surface $SL(2, \R)/SL(2, \Z)$, and study the associated first return map, which coincides with a transformation (the {\it BCZ map}) defined by Boca-Cobeli-Zaharescu~\cite{BCZ}. We classify ergodic invariant measures for this map and prove equidistribution of periodic orbits. As corollaries, we obtain results on the average depth of cusp excursions and on the distribution of gaps for Farey sequences and slopes of lattice vectors.
\end{abstract}
\maketitle

\section{Introduction}\label{sec:intro} Let $X_2 = SL(2, \R)/SL(2, \Z)$ be the space of unimodular lattices in $\R^2$. $X_2$ is a non-compact finite-volume (with respect to Haar measure on $SL(2, \R)$) homogeneous space. $X_2$ can also be viewed as the unit-tangent bundle to the hyperbolic orbifold $\h^2/SL(2, \Z)$. The action of various one-parameter subgroups of $SL(2, \R)$ on $X_2$ via left multiplication give several important examples of dynamics on homogeneous spaces, and have close links to geometry and number theory.

For example, the action of the subgroup $$A: = \left\{g_t = \left(\begin{array}{cc}e^{t/2} & 0 \\0 & e^{-t/2}\end{array}\right): t \in \R \right\}$$ yields the \emph{geodesic flow} on $X_2$, whose orbits are hyperbolic geodesics when projected to $\h^2/SL(2,\Z)$. This flow can be realized as an \emph{suspension flow} over (the natural extension of) the Gauss map $G(x) = \left\{ \frac{1}{x} \right\},$ and this connection can be exploited to give many connections between the theory of continued fractions and hyperbolic geometry (see the beautiful articles by Series~\cite{Series} or Arnoux~\cite{Arnoux} for very elegant expositions). 

In this paper, we study the \emph{horocycle flow}, that is, the action of the subgroup $$N =  \left\{h_s = \left(\begin{array}{cc}1 & 0 \\-s& 1\end{array}\right): s \in \R \right\}.$$ The main result (Theorem~\ref{theorem:main}) of this paper displays the horocycle flow as a suspension flow over the \emph{BCZ map}, which was constructed by Boca-Cobeli-Zaharescu~\cite{BCZ} in their study of Farey fractions. Equivalently, we construct a transversal to the horocycle flow so that the first return map is the BCZ map. This enables us to use well-known ergodic and equidistribution properties of the horocycle flow to derive equivalent properties for the BCZ map (in particular that it is ergodic, zero entropy, and that periodic orbits equidistribute), and give a unified explanation of several number-theoretic results on the statistical properties of Farey gaps. We also give a `piecewise-linear' description of the cusp excursions of the horocycle flow, and derive several new results on the geometry of numbers relating to gaps of slopes of lattice vectors.

\subsection{Plan of paper}\label{subsec:plan} We describe the organization of the paper, and also give a guide for readers.

\subsubsection{Organization}This paper is organized as follows: in the remainder of the introduction (\S\ref{sec:intro}), we state our results: in \S\ref{subsec:transversal}, we describe the transversal to the horocycle flow, and state Theorem~\ref{theorem:main}. In \S\ref{subsec:return}, we discuss the ergodic properties of the BCZ map (\S\ref{subsubsec:bcz}) and the structure and equidistribution of periodic orbits (\S\ref{subsubsec:period}). A piecewise-linear description of horocycle cusp excursions is given in \S\ref{subsubsec:plh}; a unified approach to Farey statistics is described in \S\ref{subsubsec:fareystat}; and a similar approach to statistics of gaps in slopes of lattice vectors is the subject of \S\ref{subsubsec:geomnumbers}. In \S\ref{sec:transversal}, we prove Theorem~\ref{theorem:main}, and show how to describe it using Euclidean and hyperbolic geometry. The structure of periodic orbits is the subject of \S\ref{sec:period}; and in \S\ref{subsec:equidist}, these structure results are used to obtain equidistribution properties and the corollaries on Farey statistics in \S\ref{sec:farey}. \S\ref{sec:ergbcz} contains the proof of ergodicity and the calculation of entropy of the BCZ map. We prove our results on cusp excursions in \S\ref{sec:plh}, and in \S\ref{sec:geomnum} we prove our results on geometry of numbers.  Finally, in \S\ref{sec:further}, we outline some questions and directions for future research.

\subsubsection{Readers guide} Since this paper touches on several different topics, it can be read in several different ways. We suggest different approaches for the ergodic-theoretic and number-theoretic minded readers. We recommend that the ergodic-theoretic reader start with sections \S\ref{subsec:transversal} and \S\ref{subsec:return} (and perhaps \S\ref{subsubsec:plh}), and follow it with \S\ref{sec:transversal}, \S\ref{sec:ergbcz} and \S\ref{sec:plh} before exploring the more number theoretic parts of the paper. The number-theoretic reader should also start with \S\ref{subsec:transversal} and \S\ref{subsec:return}, but then may be more intrigued by the results of \S\ref{subsubsec:period}, \S\ref{subsubsec:fareystat}, and \S\ref{subsubsec:geomnumbers} and their proofs in \S\ref{sec:period},\S\ref{sec:farey}, and \S\ref{sec:geomnum} respectively.

\subsection{Description of transversal}\label{subsec:transversal} Recall that $X_2$ can be explicitly identified with the space of unimodular lattices via the identification $$gSL(2, \Z) \leftrightarrow g\Z^2.$$ Let \begin{equation}\label{eq:P}P = \left\{ p_{a, b} = \left(\begin{array}{cc}a & b \\0 & a^{-1}\end{array}\right) : a \in \R^*, b \in \R\right\}\end{equation} denote the group of upper-triangular matrices in $SL(2, \R)$. Let \begin{equation}\label{eq:omega}\Om : = \left\{ p_{a, b} SL(2, \Z): a, b \in (0, 1], a+b > 1\right\} \subset X_2.\end{equation} By abuse of notation, we also use $\Om$ to denote the subset $$\{ (a, b) \in \R^2: a, b \in (0, 1], a+b >1\} \subset \R^2.$$ In \S\ref{sec:transversal}, we will show that $\Om$ can be viewed as the set of lattices with a horizontal vector of length at most $1$, can also be identified with the subset $\left\{z = x+iy \in \h^2: |x| < \frac{1}{2}, y>1, |z| >1\right\} \subset \h^2$ of the upper-half plane. Our main theorem is that the $\Om$ is a Poincar\'e section for the horocycle flow, and that the first return map is the BCZ map. We see the space $\Om$ together with the roof in Figure~\ref{fig:bczroof}.

\begin{figure}[htbp]
\begin{center}

\caption{A picture of the suspension space over $\Om$. Trajectories of the flow are vertical lines. Starting from $(a, b) \in \Om$, $h_s$ trajectories move vertically until they hit the `roof' (at time $R(a,b)$) upon which they return to the floor at position $T(a,b)$.}\label{fig:bczroof}
\includegraphics[height =80mm, width=100mm]{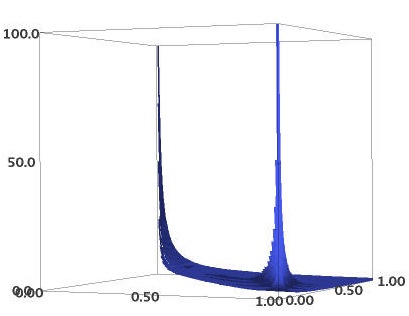}
\end{center}
\end{figure}

\begin{Theorem}\label{theorem:main} $\Om \subset X_2$ is a Poincar\'e section for the action of $N$ on $X_2$. That is, every $N$-orbit $\{h_s \La\}_{s \in \R}$ (with the exception of the codimension $1$ set of lattices $\La$ with a length $\le 1$ vertical vector), $\La \in X_2$, intersects $\Om$ in a non-empty, countable, discrete set of times $\{s_n\}_{n \in \Z}$. Given $\La_{a, b} = p_{a, b} SL(2, \Z)$, the first return time $R(a, b) = \min\{ s>0: h_s \La_{a, b} \in \Om\}$ is given by \begin{equation}\label{eq:return}R(a, b) = \frac{1}{ab}.\end{equation} The first return map $T: \Om \rightarrow \Om$ defined implicitly by $$\La_{T(a, b)} = h_{R(a,b)} x _{a, b}$$ is given explicitly by the BCZ map \begin{equation}\label{eq:bcz}T(a, b) = \left(b, -a + \left\lfloor\frac{1+a}{b}\right\rfloor b \right)\end{equation}

\end{Theorem}
\bigskip
\noindent\textbf{Remarks:} 

\begin{description}

\item[Short periodic orbits] We will see below that lattices with a length $\le 1$ vertical vector consist of those whose horocycle orbits are periodic of period at most $1$, and are embedded closed horocycles in $\h^2/SL(2, \Z)$, foliating the cusp. Thus, from a dynamical point of view, there is no loss in missing them.
\medskip
\item[Integrability of $R$] A direct calculation shows that $\int_{\Om} R 2 da db = \frac{\pi^2}{3},$ which is the volume of $X_2$ when viewed as the unit tangent bundle of $\h^2/SL(2, \Z)$ with respect to the measure $\frac{dx dy d\theta}{y^2}$. It is also immediate that $R \in L^p(dm)$ for $p<2$, where $dm = 2dadb$.
\end{description}

\subsection{Return map and applications}\label{subsec:return}
The BCZ map has proved to be a powerful technical tool in studying various statistical properties of Farey fractions~\cite{ABCZ, BCZ, BGZ}, and the distribution of angles of families of hyperbolic geodesics~\cite{BPPZ}. 
\subsubsection{Tiles and images}\label{subsubsec:tiles} We briefly describe the basic structure of the piecewise linear decomposition of $T$. (\ref{eq:bcz}) tells us that the map $T$ acts via $$T(a, b) = (a, b)A_k^T,$$ where $$A_k =  \left(\begin{array}{cc}0 & 1 \\-1 & k\end{array}\right)$$ on the region $\Om_k : = \{(a,b) \in \Om:\kappa(a,b) = k\},$ where $\kappa(a, b) = \left\lfloor\frac{1+a}{b}\right\rfloor$  (see Figure~\ref{fig:farey:triangle} below). $\Om_1$ is a triangle with vertices at $(0, 1)$, $(1, 1)$, and $\left(\frac 1 3, \frac 2 3\right)$; and for $k \geq 2$, $\Om_k$ is a quadrilateral with vertices at $\left(1, \frac 2 k \right)$, $\left(1,\frac{ 2}{k+1}\right)$, $\left( \frac{k-1}{k+1}, \frac{2}{k+1}\right)$, and  $\left( \frac{k}{k+2}, \frac{2}{k+2}\right)$. Note that $\area(\Om_1) = \frac{1}{6}$, andfor $k \geq 2$, $\area(\Om_k)= \frac{4}{k(k+1)(k+2)} = O(k^{-3})$ so that $\kappa\in L^p(m)$ for $1\le p<2$.

\begin{figure}\caption{The Farey triangle $\Om = \bigcup_{k \geq 1} \Om_k$ }\label{fig:farey:triangle}
\begin{tikzpicture}[scale=1.3]
\draw  (-4,0) -- (4,0);
\filldraw[fill=black!20!white] (4,0)--(4,4)--(0,4)--cycle;
\draw (0,0)--(0,4)node[above]{\tiny $(0,1)$};

\foreach \x in {2, 3, 4, 5, 6, 7} \draw[dashed] (-4,-0) -- (4, 8/\x)node[right]{\tiny $2/\x$};

\path (-4, 0) node[below]{\tiny $(-1,0)$};
\path (4, 0) node[right]{\tiny $(1,0)$};

\path (2, 3.5) node{\tiny $\Om_1$};
\path (8/3, 8/3) node{\tiny $\Om_2$};
\path (3, 2) node{\tiny $\Om_3$};
\path (16/5, 8/5) node{\tiny$\Om_4$};

\path (16/5, 7/5) node[right]{$\ddots$};

\path (3.8, .8) node{\tiny $\Om_{\geq 7}$};

\end{tikzpicture}
\end{figure}
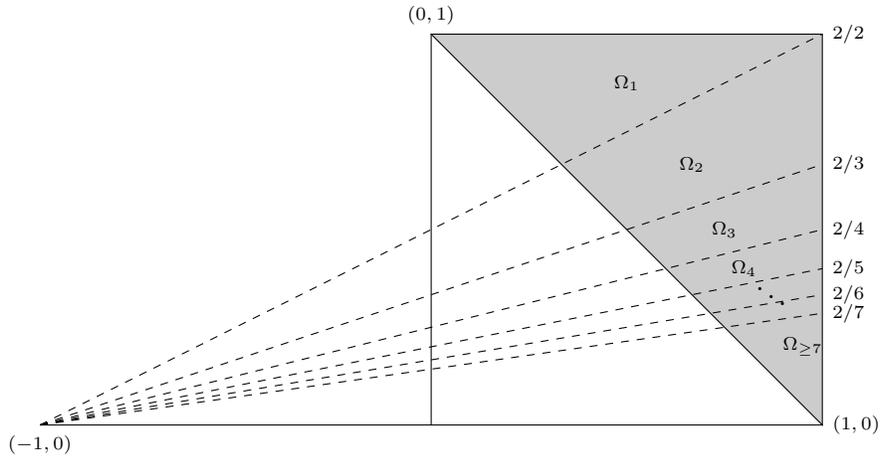
The image $T\Om_k$ is the reflection of $\Om_k$ about the line $a = b$ (see Figure~\ref{fig:farey:image}). However, $T$ is \emph{orientation-preserving}, and thus does not act by the reflection. For $k=1$, $T$ acts by the elliptic matrix $A_1$, for $k=2$ the parabolic matrix $A_2$, and for $k \geq 3$, by the hyperbolic matrices, since $\mbox{trace}(A_k) = k.$

\begin{figure}\caption{The images $T\Om_k$}\label{fig:farey:image}
\begin{tikzpicture}[scale=1.3]
\draw  (0, -4) -- (0,4);
\filldraw[fill=black!20!white] (4,0)--(4,4)--(0,4)--cycle;
\draw (0,0)--(0,4);
\draw (0,0)--(4,0);

\foreach \x in {2, 3, 4, 5, 6, 7} \draw[dashed] (0,-4) -- (8/\x,4)node[above]{\tiny $\frac{2}{\x}$};

\path (3.5, 2) node{\tiny $T\Om_1$};
\path (8/3, 8/3) node{\tiny $T\Om_2$};
\path (2, 3) node{\tiny $T\Om_3$};
\path (8/5, 3.2) node{\tiny$T\Om_4$};
\path (9/7, 21/6) node{$\ddots$};
\path (.8, 3.6) node{\tiny $T\Om_{\geq 7}$};

\path (0, -4) node[below]{\tiny $(0, -1)$};
\path (4, 0) node[right]{\tiny $(1,0)$};

\path (0, 4) node[left]{\tiny $(0,1)$};

\end{tikzpicture}
\end{figure}
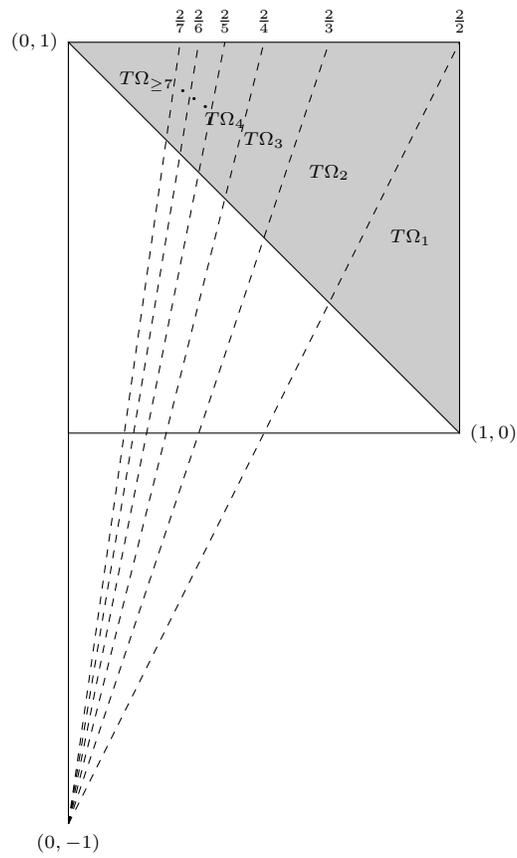

\subsubsection{Ergodic properties of the BCZ map}\label{subsubsec:bcz} In~\cite[\S3]{BZsurvey}, Boca-Zaharescu posed a series of questions on the ergodic properties of $T$:

\begin{question*} Is $T$ ergodic (with respect to Lebesgue measure)? Weak mixing? What is the entropy of $T$?\end{question*}

\noindent A corollary of Theorem~\ref{theorem:main} and the ergodicity and entropy properties of the horocycle flow is:

\begin{Theorem}\label{theorem:BCZ:ergodic} $T$ is an ergodic, zero-entropy map with respect to the Lebesgue probability measure $dm = 2dadb$. Moreover, $m$ is the unique absolutely continuous invariant probability measure, and in fact is the unique ergodic invariant measure not supported on a periodic orbit.\end{Theorem}

\subsubsection{Periodic orbits}\label{subsubsec:period}The map $T$ has a rich and intricate structure of periodic orbits, closely related to the Farey sequences.  A direct calculation shows that for $Q \in \N$  the point $\La_{\frac1 Q, 1} = p_{ \frac{1}{Q}, 1} \Z^2$ is $h_s$-periodic with (minimal) period $Q^2$.

Given $Q \in \N$, the Farey sequence $\F(Q)$ is the collection (in increasing order) of fractions $0 < \frac{p}{q} < 1$ with $q \le Q$. Let $N = N(Q)$ denote the cardinality of $Q$. We write $\F(Q) = \{ \frac{0}{1} = \gamma_1 < \gamma_2 = \frac{1}{Q} < \ldots <\gamma_{N}\}$. For notational convenience, we write $\gamma_{N +1} = 1 = \frac{1}{1}$. Writing $\gamma_i = \frac{p_i}{q_i}$, with $q_i \le Q$, we have $q_i + q_{i+1} > Q$, and $$a_{i+1} q_i - p_i q_{i+1} = 1.$$ The following fundamental observation is due to Boca-Cobeli-Zaharescu~\cite{BCZ}:
\begin{equation}\label{eq:farey:bcz} T\left(\frac{q_i}{Q}, \frac{q_{i+1}}{Q}\right) = \left( \frac{q_{i+1}}{Q}, \frac{q_{i+2}}{Q}\right).
\end{equation}
The indices are viewed cyclically in $\Z/N\Z$. We give a geometric explanation for (\ref{eq:farey:bcz}) in \S\ref{subsec:fareyperiod} below. Thus, $( \frac{1}{Q}, 1)$ is a periodic point of order $N$ for $T$. To relate this to the periodic orbit of $\La_{1, \frac{1}{Q}}$ for $h_s$, we first record two simple observations. First, $$\left(\frac{q_1}{Q}, \frac{q_2}{Q}\right) = \left(\frac{1}{Q}, 1\right),$$ and second, $$\sum_{i=1}^{N} \frac{1}{q_i q_{i+1}} =\sum_{i=1}^{N} (\gamma_{i+1} -\gamma_{i} ) =  1$$ Multiplying both sides by $Q^2$, and applying (\ref{eq:farey:bcz}), we obtain $$\sum_{i=1}^{N} \frac{Q^2}{q_i q_{i+1}} = \sum_{i=1}^{N} \left(R\circ T^{i-1}\right)\left(\frac{1}{Q}, 1\right) = Q^2.$$ Thus, as we sum the roof function $R$ over the periodic orbit for $T$, we obtain the length of the associated periodic orbit for the flow.

We recall Sarnak~\cite{Sarnak} showed that long periodic orbits for $h_s$ (i.e., long closed horocycles) become equidstributed with respect to $\mu_2$, the Haar measure on $X_2$. Our main result on the distribution of periodic orbits is the following discrete corollary of Sarnak's result: let $\rho_{Q, I}$ denote the probability measure supported on (a long piece of) the orbit of $(\frac{1}{Q}, 1)$: given $I = [\alpha, \beta] \subset [0, 1]$, let $N_I(Q): = |\F(Q) \cap I|$, and define $$\rho_{Q, I} = \frac{1}{N_I(Q)}\sum_{i: \gamma_i \in I} \delta_{T^i( \frac{1}{Q}, 1)}.$$ If $I = [0, 1]$, we write $\rho_{Q, I} = \rho_Q$.

\begin{Theorem}\label{theorem:equidist:periodic} For any (non-empty) interval $I =[\alpha, \beta] \subset [0, 1]$, the measures $\rho_{Q, I}$ become equidistributed with respect to $m$ as $Q \rightarrow \infty$. That is, $\rho_{Q, I} \rightarrow m$, where convergence is in the weak-$^*$ topology.\end{Theorem}

\noindent\textit{Remark.}  The case $I=[0,1]$ of Theorem~\ref{theorem:equidist:periodic} was proven 
in \cite{KZ} using different methods.  Also, Theorem~\ref{theorem:equidist:periodic} can be 
deduced from Theorem 6 in \cite{Marklof1}, which is proven in the more general context 
of horospherical flows using a section that reduces to the same one in Theorem~\ref{theorem:main}.

\begin{Theorem}\label{theorem:dense:periodic} A point $(a,b) \in \Om$ is $T$-periodic if and only if it has rational slope. In particular, the set of periodic points is dense. \end{Theorem}

\noindent While periodic points are abundant, there are strong restrictions on the \emph{lengths} of periodic orbits, and the associated matrices, governed by the relationship between the (discrete) period $P(a,b)$ under $T$ of a point $(a,b)$ and the (continuous) period $s(a,b)$ of the associated lattice $p_{a,b} \Z^2$ under $h_s$. We fix notation: for $n \geq 1$, we write $$A_n(a,b) = A\left(T^{n-1}(a,b)\right) A\left(T^{n-2}(a,b)\right) \ldots A(a,b),$$ so $$T^n(a,b) = (a,b)A_n(a,b)^T.$$

As a starting point for our observations, note that the diagonal $(a,a) \in \Om$ consists of \emph{fixed} (i.e, period $P(a,a) = 1$) points for $T$, the associated horocycle period is the value roof function $R(a,a) = \frac{1}{x^2}$, and the associated matrix $$A_{P(a,a)} (a,a) =  A(a,a) = A_2 =  \left(\begin{array}{cc}0 & 1 \\-1 & 2\end{array}\right)$$ is parabolic. Our main result on periodic points is that appropriate versions of these observations hold for all (segments of) periodic points.

\begin{Theorem}\label{theorem:structure:periodic} For any periodic point $(a, b) \in \Om$, we have $$P(a,b) = N\left(\left\lfloor \sqrt{s(a,b)} \right \rfloor\right).$$ Thus the set of possible periods is given by the cardinalities $\{N(Q): Q \in \N\}$ of Farey sequences. Moreover, the matrix $$A_{P(a,b)}(a,b) = A\left(T^{P(a,b)-1}(a,b)\right)A\left(T^{P(a,b)-2}(a,b)\right) \ldots A(a,b)$$ associated to the periodic orbit is always parabolic , and in fact constant along the segment $\left\{(ta,tb): t \in \left(\frac{1}{a+b}, 1\right]\right\}$.  Precisely, for $k, l \in \N$, $k \le l$ relatively prime, and $a \in \left (\frac{l}{l+r}, \frac{l}{l+r-1}\right]$, $1 \le r \le k$, we have $$P\left(a, a\frac{k}{l}\right) = N(l + r -1),$$ and for $a \in  (\frac{l}{l+k}, 1]$, $$A_{P\left(a, a\frac{k}{l}\right)}\left(a, a\frac{k}{l}\right) = \left(\begin{array}{cc}1-kl & l^2 \\-k^2 & 1+kl\end{array}\right).$$
\end{Theorem}

\subsection{Piecewise linear description of horocycles}\label{subsubsec:plh}In this section, we describe the results that originally motivated this project, on cusp excursions and returns to compact sets for the horocycle flow. Let $\| .\|$ denote the supremum norm on $\R^2$. For $\La = g\Z^2 \in X_2$, define $$\ell(\La) : = \inf_{0 \neq \vv \in g\Z^2} \|\vv\|,$$ and let $\alpha: X_2 \rightarrow \R^+$ be given by $$\alpha(\La) = \frac{1}{\ell(\La)}.$$ By Mahler's compactness criterion, a subset $A$ of $X_2$ is precompact if and only if there is an $\epsilon >0$ so that for all $\La \in A$, $$\ell(\La) > \epsilon,$$ or, equivalently if $\alpha|_A$ is bounded.

Dani~\cite{Dani} showed that for any lattice $\La \in X_2$, the orbit $\{h_s \La\}_{s \geq 0}$ is either closed or uniformly distributed with respect to the Haar probability measure $\mu$ on $X_2$. Thus for an $\La$ so that $\{h_s \La\}_{s \geq 0}$ is not closed, we have $$\limsup_{s \rightarrow \infty} \alpha_1(h_s \La) = \infty$$ and $$\limsup_{s \rightarrow \infty} \ell(h_s \La) = 0.$$ In fact, this does not require equidistribution but simply density (due to Hedlund~\cite{Hedlund}). $\{h_s \La\}$ is closed if and only if $\La$ has vertical vectors.

A natural question is to understand the \emph{rate} at which these excursions to the non-compact part (`cusp') of $X_2$ occur. The first-named author, in joint work with G.~Margulis~\cite{AM1}, showed that for any $\La \in X_2$ without vertical vectors, $$\limsup_{s \rightarrow \infty} \frac{\log \alpha_1(h_s \La)}{\log s} \geq 1,$$ and related the precise limit to Diophantine properties of the lattice $\La$ (see also~\cite{A}).

Our results concern the average behavior of \emph{all} visits to the cusp, as defined by local minima of the function $\ell_{\La}(s) = \ell(h_s \La)$ (or, equivalently, local maxima of $\alpha_{\La}(s) = \alpha(h_s\La)$). The function $\ell_{\La}(s)$  is a piecewise-linear function of $s$, and helps give a picture of the`height' of the horocycle at time $s$. In this way it is similar to work of the second author~\cite{Cheung}, where a piecewise linear description of diagonal flows on $SL(3, \R)/SL(3, \Z)$ was studied.

We also consider returns to the compact part of $X_2$, given by local minimal of $\ell_{\La}(s)$. Given $\La \in X_2$ so that $\{h_s \La\}_{s \geq 0}$ is not closed, let $\{s_n\}$ and $\{S_n\}$ denote the sequences of minima and maxima of $\ell_{\La}(s)$ respectively.  To be completely precise, the local minima for the supremum norm occur in intervals, and we take $s_n$ to be the \emph{midpoint} of the $n^{th}$ interval. Define the averages $$a_N(\La): = \frac{1}{N}\sum_{n=1}^N \alpha_{\La}(s_n)\mbox{ ; }A_N(\La) := \frac{1}{N}\sum_{n=1}^N \alpha_{\La}(S_n),$$ $$l_N(\La) :=  \frac{1}{N}\sum_{n=1}^N \ell_{\La}(s_n)\mbox{ ; } L_N(\La) := \frac{1}{N}\sum_{n=1}^N \ell_{\La}(S_n).$$

\begin{Theorem}\label{theorem:minima} For any $\La$ without vertical vectors, we have
\begin{equation}\label{equation:minima:alpha} \lim_{N \rightarrow \infty} a_N(\La) = 2 \end{equation}
\begin{equation}\label{equation:minima:ell} \lim_{N \rightarrow \infty} l_N(\La) = \frac{2}{3} \end{equation}
\end{Theorem}
\medskip
\noindent Let $M: \Omega \rightarrow [0, 1]$ be given by \begin{equation}\label{eq:max}M(a,b) = \max\left\{a, b, \frac{1}{a+b}\right\}\end{equation}

\begin{Theorem}\label{theorem:maxima} For any $\La$ without vertical vectors, we have

\begin{equation}\label{equation:maxima:alpha} \lim_{N \rightarrow \infty} A_N(\La) = \int_{\Om} \frac{1}{M} dm  = \frac{2}{3}\left(13-8\sqrt 2\right) \approx 1.2\end{equation}
\begin{equation}\label{equation:maxima:ell} \lim_{N \rightarrow \infty} L_N(\La) = \int_{\Om} Mdm = \frac{2}{3}\left(7-4\sqrt 2\right) \approx .73\end{equation}
\end{Theorem}
\medskip
\noindent Theorem~\ref{theorem:minima} and Theorem~\ref{theorem:maxima} follow from applying the ergodic theorem to the BCZ transformation $T: \Omega \rightarrow \Omega$. The assumption that $\La$ does not have vertical vectors is to ensure it does not have a periodic orbit under $T$ (equivalently, $h_s$). There is a version of this result for periodic orbits (Corollary~\ref{cor:farey:amusing}) given in \S\ref{subsubsec:excursions} 
The function $M$ gives the maximum of the function $\ell$ on a sojourn from the transversal $\Omega$. The limits (\ref{equation:minima:alpha}) and (\ref{equation:minima:ell}) in Theorem~\ref{theorem:minima} are the integrals of the functions $g(a,b) = \La$ and $h(a,b) = \frac{1}{\La}$ over $\Omega$ respectively. The times $s_n$ correspond to the return times of $\{h_s\La\}$ to $\Om$.

\subsection{Farey Statistics}\label{subsubsec:fareystat}Theorem~\ref{theorem:equidist:periodic} has several number theoretic corollaries. We record results on spacing and indices of Farey fractions, originally proved using analytic methods in~\cite{ABCZ, BCZ, Hall}.  Our results give a unified explanation for these equidistribution phenomena.  For similar applications in the context of higher dimensional generalizations of Farey sequences we refer the reader to \cite{MS} and \cite{Marklof2}.  

\subsubsection{Spacings and $h$-spacings}\label{subsubsec:Hall} We now fix an interval $I \subset [0, 1]$. It is well known that the Farey sequences $\F_I(Q)$ become equidistributed in $[0, 1)$ as $Q \rightarrow \infty$. A natural statistical question is to understand the distribution of the \emph{spacings} $(\gamma_{i+1} - \gamma_{i})$. Since there are $N_I(Q)$ points, and $N_I(Q) \sim |I| \frac{3}{\pi^2} Q^2$, the natural normalization yields the following question: given $0 \le c \le d$, what is the limiting ($Q \rightarrow \infty$) behavior of $$\frac{ |\{\gamma_{i} \in \F_ {I}(Q): \frac{3}{\pi^2} |I| Q^2(\gamma_{i+1} - \gamma_{i}) \in [c, d]\}|}{N_I(Q)} ?$$

\noindent Since $Q^2(\gamma_{i+1} - \gamma_i)  = \left(R\circ T^{i-1}\right)\left(\frac{1}{Q}, 1\right)$, the above quantity reduces to \begin{equation}\label{eq:spacing}\rho_{Q, I}\left ( R^{-1}\left(\left[\frac{\pi^2}{3|I|}c, \frac{\pi^2}{3|I|}d\right]\right)\right).\end{equation} Since $R^{-1}\left(\left[\frac{\pi^2}{3|I|}c, \frac{\pi^2}{3II|}d\right]\right)$ is \emph{compact}, we can directly apply Theorem~\ref{theorem:equidist:periodic} to (\ref{eq:spacing}) obtain a result of R.~R.~Hall~\cite{Hall}:

\begin{Cor}\label{cor:hall} As $Q \rightarrow \infty$, $$ \frac{ |\{\gamma_i \in \F_I(Q): \frac{3}{\pi^2}|I|Q^2(\gamma_{i+1} - \gamma_{i}) \in (c, d)\}|}{N_I(Q)} \rightarrow m\left (R^{-1}\left(\frac{\pi^2}{3|I|}c, \frac{\pi^2}{3|I|}d\right)\right). $$
\end{Cor}

\noindent We call this distribution \emph{Hall's distribution}. A generalization of this result to higher-order spacings was considered by Augustin-Boca-Cobeli-Zaharescu~\cite{ABCZ}. They considered \emph{$h$-spacings}: the vector of $h$-tuples ($h \geq 1$) of spacings $\vv_{i, h} = \vv_{i,h}(Q)  = \left(\gamma_{i+j} - \gamma_{i+j-1}\right)_{j=1}^{h} \in \R^h$. Given $\B = \prod_{i=1}^{k} \left[c_i, d_i\right]$, we have, as above, \begin{equation}\label{eq:hspacing}\frac{ \left|\left\{\gamma_i \in \F_I(Q): \frac{3}{\pi^2}|I|Q^2\vv_{i,h} \in \B\right\}\right|}{N_I(Q)} = \rho_{Q, I}\left ( R_h^{-1}\left(\tilde{B}\right)\right),\end{equation} where $R_h: \Om \rightarrow \R^h$ is given by $$R^h(a, b) = \left(R\left(T^{j-1}\left(a,b\right)\right)\right)_{j=1}^{h},$$ and $$\tilde{B} = \frac{\pi^2}{3|I|}B = \prod_{i=1}^{k} \left[\frac{\pi^2}{3|I|}c_i, \frac{\pi^2}{3|I|}d_i\right].$$ Applying our equidistribution result to (\ref{eq:hspacing}), we recover Theorem 1.1 of~\cite{ABCZ}:

\begin{Cor}\label{cor:ABCZ} As $Q \rightarrow \infty$, $$\frac{ |\{\gamma_i \in \F_I(Q): \frac{3}{\pi^2}Q^2\vv_{i,h} \in \B\}|}{N_I(Q)} \rightarrow m\left ( R_h^{-1}\left(\tilde{B}\right)\right). $$
\end{Cor}

\subsubsection{Density of the limiting distribution}\label{subsubsec:density} Corollary~\ref{cor:hall} allows one to explicitly calculate the density of the limiting distribution of consecutive gaps, that is, of Hall's distribution. For $d >0$, we define the distribution function $G_c(d)$ to be the asymptotic proportion of (normalized) gaps of size at most $b$, that is, $$G_c(d) = \lim_{Q \rightarrow \infty}  \frac{ |\{\gamma_i \in \F_I(Q): \frac{3}{\pi^2}|I|Q^2(\gamma_{i+1} - \gamma_{i}) \in (0, d)\}|}{N_I(Q)}.$$ By Corollary~\ref{cor:hall}, we can write $$G_c(d) = m\left (R^{-1}\left(0, \frac{\pi^2}{3|I|}d\right)\right).$$ Since for $d < \frac{3|I|}{\pi^2}$, the curve $R(a,b) =  \frac{\pi^2}{3|I|}d$ does not intersect $\Om$, we have $G_c(d) = 0$. In particular, the distribution does not have any support at $0$. For $d \geq \frac{3|I|}{\pi^2},$ $G_c(d) >0$. The distribution has another point of non-differentiability at $d = 4\frac{3}{\pi^2},$ when the curve $R(a, b) = \frac{3|I|}{\pi^2}d$ intersects the bottom line $a+b=1$. The picture of the distribution for $|I| = [0, 1)$ is given in Figure~\ref{fig:bczdist}, reproduced from~\cite{BZsurvey}. In Figure~\ref{fig:bczregion}, we display a picture of the region $R^{-1}([c, d])$.

\begin{figure}\caption{The distribution function for the gaps}

\begin{center}
  \includegraphics[width=0.5\textwidth, height = 50mm]{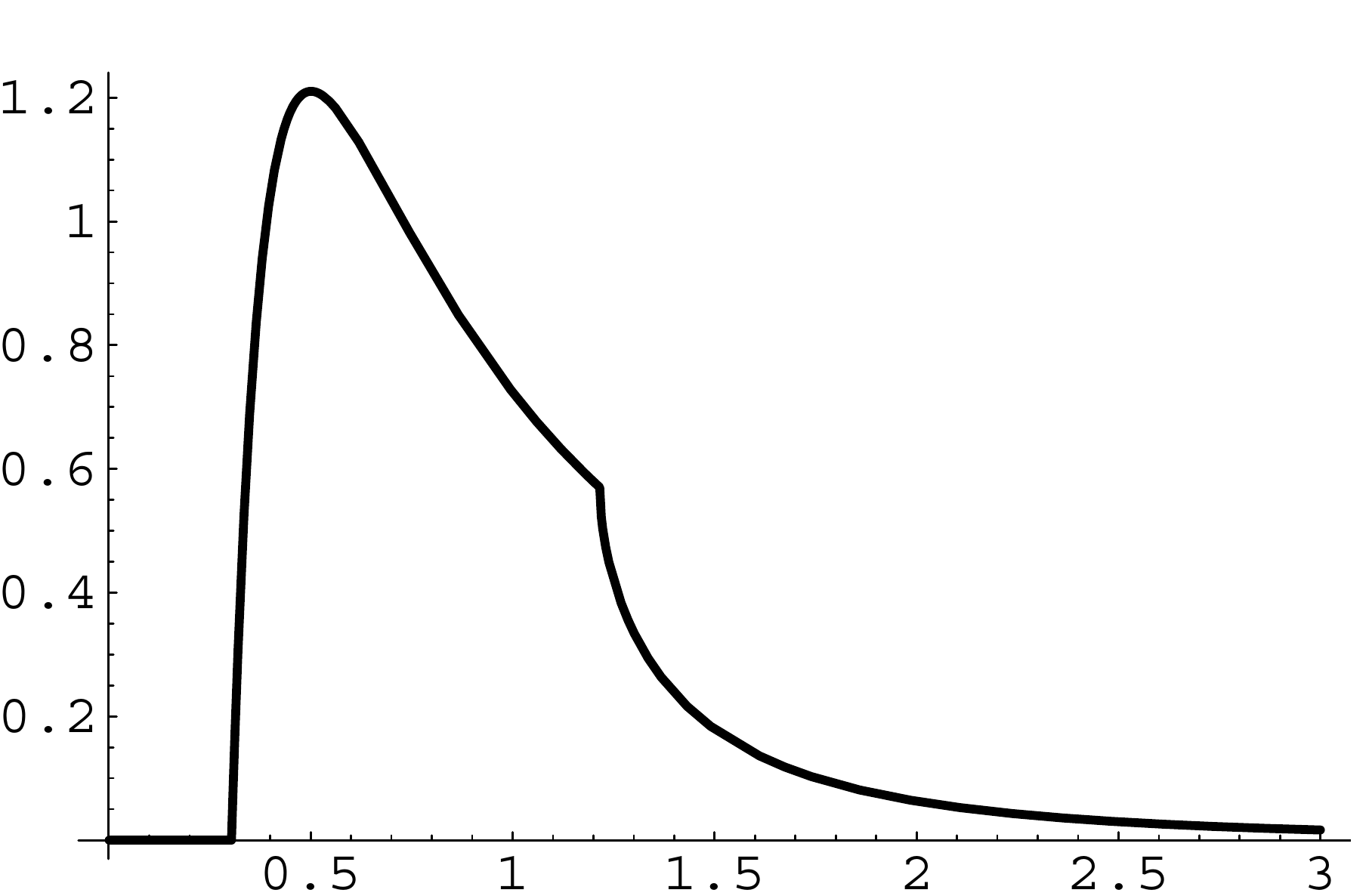}
\end{center}\label{fig:bczdist}

\end{figure}

\begin{figure}[htbp]\caption{The region \textcolor{yellow}{$R^{-1}([c, d])$} $\subset$ \textcolor{blue}{$\Om$}.\medskip}\label{fig:bczregion}
    \includegraphics[width=50mm, height = 50mm]{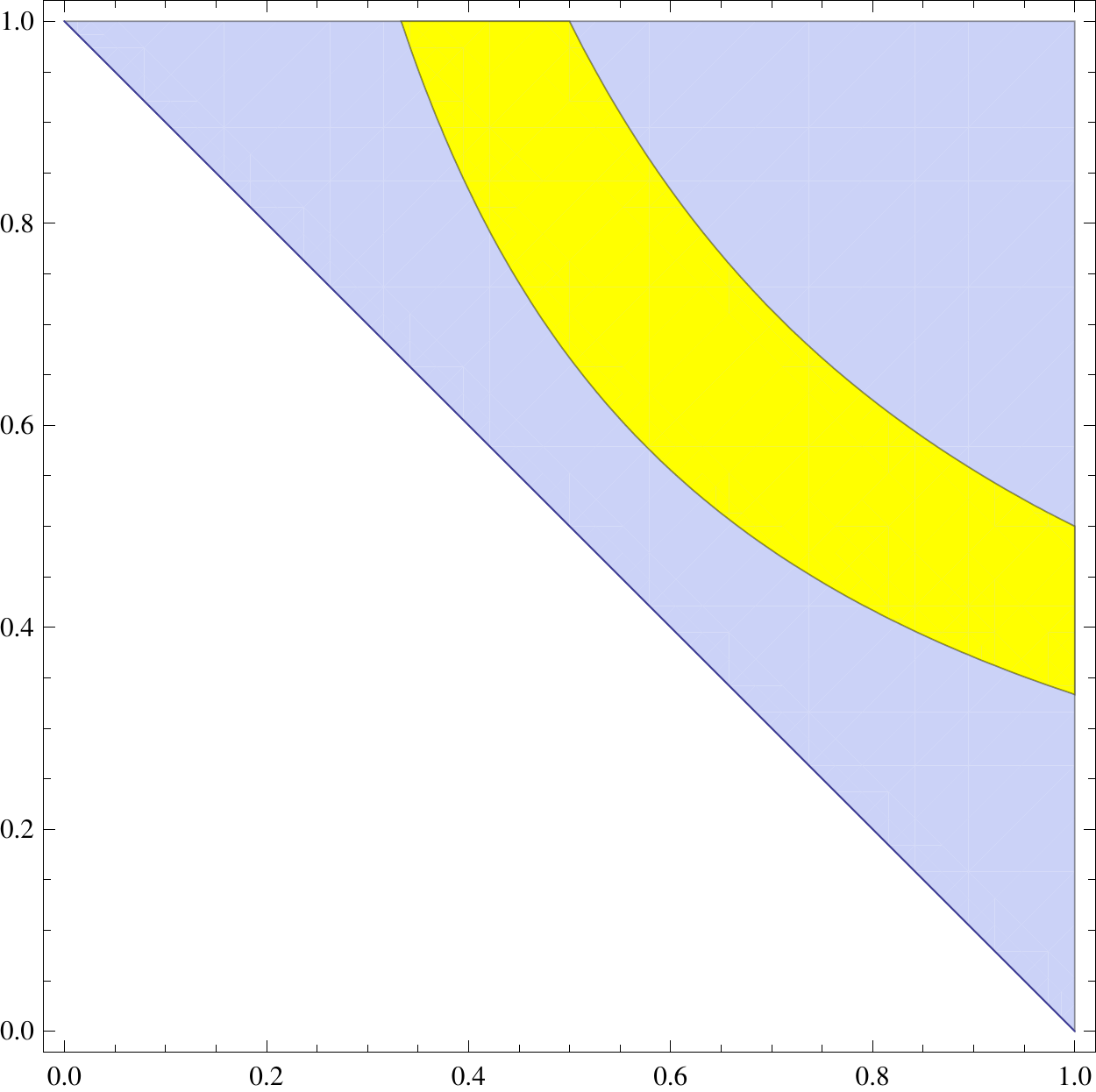}

\end{figure}

\subsubsection{Indices of Farey fractions}\label{subsubsec:Index} Following~\cite{Hall-Shiu}, we define the \emph{index} of the Farey fraction $\gamma_i = \frac{a_i}{q_i}\in \F(Q)$ by $$\nu(\gamma_i) : = \frac{q_{i-1} + q_{i+1}}{q_i} =\left\lfloor \frac{Q+q_{i-1}}{q_i} \right\rfloor.$$ In~\cite{BCZ}, this was reinterpreted in terms of the map $T$ and the function $\kappa: \Om \rightarrow \N$. Recall that  $\kappa(x, y) = \left\lfloor \frac{1+x}{y} \right\rfloor$. Precisely, we have $$\nu(\gamma_i) = \kappa\left(\frac{q_{i-1}}{Q}, \frac{q_i}{Q}\right) = \left(\kappa \circ T^i\right)\left(\frac{1}{Q}, 1\right).$$

\noindent Fix a (non-trivial) interval $I \subset [0, 1)$, and an exponent $\alpha \in (0, 2)$. We consider the average $$\rho_{Q, I} ( \kappa^{\alpha}) = \frac{1}{N_I(Q)}\sum_{\gamma_i \in \F_I(Q)} \nu(\gamma_i)^{\alpha}.$$ Applying Theorem~\ref{theorem:equidist:periodic} to $\kappa^{\alpha}$ (note that $\alpha \in (0,2)$ implies that $\kappa^{\alpha} \in L^1$), we obtain a result originally due to Boca-Gologan-Zaharescu~\cite{BGZ}:

\begin{Cor}\label{cor:BGZ} As $Q \rightarrow \infty$, $$ \frac{1}{N_I(Q)}\sum_{\gamma_i \in \F_I(Q)} \nu(\gamma_i)^{\alpha} \rightarrow  \int_{\Om} \kappa^{\alpha} dm = 2 \sum_{k=1}^{\infty} k^{\alpha} \area(\Om_k) . $$ Here $\Om_k = \kappa^{-1}(k)$.
\end{Cor}

\subsubsection{Powers of denominators}\label{subsec:classical} We can also obtain general results on sums associated to Farey fractions by applying Theorem~\ref{theorem:equidist:periodic} to various functions in $L^1(\Om, m)$. For example, if we define $f_{s, t}: \Om \rightarrow \C$ by $$f_{s,t}(a,b) = a^s b^t,$$ for $s, t \in \C$ with $\Re s, \Re t  > -1$, and define $$B_{s, t} := \|f_{s, t} \|_1 = \int_{\Om} f_{s,t}(x,y) dm = 2\left(\frac{1}{(s+1)(t+1)} -\frac{\Gamma(s+1)\Gamma(t+1)}{\Gamma(s+t+3)}\right),$$ we obtain results originally due to Hall-Tanenbaum~\cite{HT}.

\begin{Theorem}\label{theorem:classic} Fix a non-trivial interval $I \subset [0, 1)$. Let $s, t \in \C$. Then, if $\Re s, \Re t > -1$, \begin{equation}\label{equation:st}\lim_{Q \rightarrow \infty}  \frac{1}{N_I(Q) Q^{s+t}}\sum_{\gamma_i \in \F_I(Q)} q_i^s q_{i+1}^t = \int_{\Om} x^s y^t dm = B_{s, t}.\end{equation} In particular, for $s = 1$ and $t=0$, we have \begin{equation}\label{equation:qi:Q}\lim_{Q \rightarrow \infty}  \frac{1}{N_I(Q)}\sum_{\gamma_i \in \F_I(Q)} \frac{q_i}{Q} = \int_{\Om} x dm = \frac{2}{3}.\end{equation} In addition, for $s = -1$ and $t =0$, \begin{equation}\label{equation:Q:qi}\lim_{Q \rightarrow \infty}  \frac{1}{N_I(Q)}\sum_{\gamma_i \in \F_I(Q)} \frac {Q }{q_i} =  \int_{\Om} \frac{1}{x} dm = 2,\end{equation} and for $s = t = -1$, we have \begin{equation}\label{equation:classic}\lim_{Q \rightarrow \infty}  \frac{Q^2}{N_I(Q)}\sum_{\gamma_i \in \F_I(Q)} \frac{1}{q_i q_{i+1}} = \int_{\Om} \frac{1}{xy} dm = \frac{\pi^2}{3},\end{equation} which yields the classical result $$N_I(Q) \sim |I| \frac{3}{\pi^2} Q^2.$$
\end{Theorem}

\subsubsection{Excursions}\label{subsubsec:excursions} Finally, we remark that applying the equisitribution result Theorem~\ref{theorem:equidist:periodic} to the functions $M$ (and $\frac{1}{M}$) defined in (\ref{eq:max}) above, we obtain an amusing statistical result on Farey fractions:

\begin{Cor}\label{cor:farey:amusing} Let $I \subset [0, 1)$ be a non-trivial interval. Then

$$ \lim_{Q \rightarrow \infty} \frac{1}{N_I(Q)}\sum_{\gamma_i \in \F_I(Q)} \min \left(\frac{Q}{q_i}, \frac{Q}{q_{i+1}}, \frac{q_i + q_{i+1}}{Q} \right) =  \frac{2}{3}\left(13-8\sqrt 2\right) $$

$$ \lim_{Q \rightarrow \infty} \frac{1}{N_I(Q)}\sum_{\gamma_i \in \F_I(Q)} \max \left(\frac{q_i}{Q}, \frac{q_{i+1}}{Q}, \frac{Q}{q_i + q_{i+1}} \right) =  \frac{2}{3}\left(7-4\sqrt 2\right)$$

\end{Cor}

\subsection{Geometry of numbers}\label{subsubsec:geomnumbers} Let $\La \in X_2$ be a unimodular lattice, and fix $t >0$. Let $$S_t(\La) = \{s_1 < s_2 < \ldots < s_n< \ldots\}$$ denote the slopes of the lattice vectors in the vertical strip $V_t \subset \R^2$ given by $$V_t: =\{(x,y): x \in (0, t], y >0\},$$ written in increasing order.  Let $$G_{N,t}(\La) = \{ s_{n+1} - s_n: 0 \le n \le N\}$$ denote the sequence of gaps in this sequence, viewed as a \emph{set}, so $|G_{N,t} (\La)| \le N$. Our main geometry of numbers result states that for lattices $\La$ without vertical vectors, this sequences has the same limiting distribution as the gaps for Farey fractions, namely, Hall's distribution. That is:

\begin{Theorem}\label{theorem:generic:lattice} Suppose $\La$ does not have vertical vectors, and let $0 \le c \le d \le \infty$. Then $$\lim_{N \rightarrow \infty} \frac{1}{N} | G_{N, t}(\La) \cap (c, d)| = 2 m(R^{-1}(c,d)).$$
\end{Theorem}
\noindent This result will follow from the application of the Birkhoff ergodic theorem to the orbit of $\La$ under the BCZ map, with observable given by the indicator function of the set $R^{-1}(c, d)$. If $\La$ does have a vertical vector, and thus a periodic orbit under the BCZ map, we have that for any $t>0$, there is an $N_0 = N_0(t) >0$ so that $G_{N,t}(\La) = G_{N_0,t} (\La)$ (as sets of numbers) for all $N > N_0$. We can use Theorem~\ref{theorem:equidist:periodic} to obtain the following:
\begin{Cor}\label{cor:periodic:lattice:gap} Let $0 \le c \le d \le \infty$. Then $$\lim_{t \rightarrow \infty} \frac{1}{N_0(t)} | G_{N_0(t), t}(\La) \cap (c, d)| = 2 m(R^{-1}(c,d)).$$
\end{Cor}

\section{Construction of Transversal}\label{sec:transversal} In this section, we prove Theorem~\ref{theorem:main}. We first give an interpretation of the transversal $\Om$ in terms of geometry of numbers in \S\ref{subsec:short}; and prove Theorem~\ref{theorem:main} in \S\ref{subsec:proof:theorem:main}. We also record some observations on slopes in \S\ref{sec:obs:slope}. We show how $\Om$ can be interpreted in terms of hyperbolic geometry in \S\ref{subsec:hyp}. In \S\ref{subsec:selfsim}, we show a certain self-similarity property of the BCZ map $T$.

\subsection{Short horizontal vectors}\label{subsec:short} Fix $t>0$. We say a lattice $\La \in X_2$ is $t$-\emph{horizontally short} if it contains a non-zero horizontal vector $\vv = (a, 0)^T$ so that $|a| \le t$. Note that since $\La$ is a group, we can assume $a > 0$. We will call $1$-horizontally short lattices simply \emph{horizontally short}. We also have the analgous notions of $t$-vertically and vertically short.To prove Theorem~\ref{theorem:main}, we break it up into several lemmas. Our first lemma is:

\begin{lemma}\label{lem:short} $\Om = \{\La_{a,b}: a, b \in (0, 1], a+b>1\}$ is the set of horizontally short lattices. \end{lemma}

\begin{proof} Since the horizontal vector $(a,0)^T$ is in $\La_{a,b}$, and $a \le 1$, clearly every lattice in $\Om$ is short. To show the reverse containment, suppose $\La$ is short. Then we can write $\La = \La_{a, b'} = p_{a, b'} \Z^2$, with $0 < a \le 1$, and $b' \neq 0$. Let $m \in \Z$ be such that $$1-a < ma + b' \le 1,$$ i.e., $m = \left\lfloor \frac {1-a}{b'} \right\rfloor$. Set $b = ma + b'$. Then, since $p_{1,m} \in SL(2,\Z)$, $$\La = \La_{a, b'} = p_{a, b'}\Z^2= p_{a, b'} p_{1,m} \Z^2 = \La_{a,b} = \La_{a,b} \in \Om.$$
\end{proof}

\noindent Next, we show that for any $(a,b) \in \Omega$, that there is a $s \in (0, \infty)$ so that $h_s \La_{a,b} \in \Om$, and give a formula for the minimum $s$. 

\begin{lemma}\label{lemma:roof} Let $(a,b) \in \Om$, so $\La_{a,b} \in \Om$. Let $s_0 = R(a,b) = \frac{1}{ab}$. Then $h_{s_0} \La \in \Om$, and for every $0 < s < s_0$, $h_s  \La_{a,b} \notin \Omega$. Furthermore, $$h_{s_0}\La_{a,b}= \La_{T(a,b)}.$$
\end{lemma}

\begin{proof} Since $\Om$ consists of horizontally short lattices, we first note that the first time $h_s \La_{a,b}$ will have a horizontal vector is given by the equation $$-sb + \frac{1}{a} = 0.$$ Thus we set $s_0 = \frac{1}{ab}$, and a direct calculation shows $$ h_{s_0} p_{a,b} =  \left(\begin{array}{cc}a & b \\-\frac{1}{b}& 0\end{array}\right).$$ Let $\kappa(a,b) = \lfloor \frac{1+a}{b} \rfloor$. Applying the matrix $$ \left(\begin{array}{cc}0 & -1 \\1 & \kappa(a,b)\end{array}\right) =\left( A(a,b)^{-1}\right)^T$$ we obtain \begin{equation}\label{eq:cocycle}h_{s_0} p_{a,b}\left( A(a,b)^{-1}\right)^T= \left(\begin{array}{cc}b & -a + \kappa(a,b) b \\0 & b^{-1}\end{array}\right)= p_{T(a,b)}.\end{equation} Thus, $h_{s_0}\La_{a,b} = \La_{T(a,b)}$, as desired.
\end{proof}

\medskip

\noindent Finally, we show that for any lattice $\La \in X_2$ which is not vertically short, the orbit under $\{h_s\}_{s \geq 0}$ will intersect $\Om$.

\begin{lemma}\label{lemma:poincare} Let $\La \in X_2$ so that $\La$ is not vertically short. Then there is a $s_1 \in \R$ so that $h_{s_1} \La \in \Om$.
\end{lemma}
\begin{proof} Let $S = [-1, 1]^2 \subset \R^2$. $S$ is a square centered at $0$, so is convex, centrally symmetric, and has area $4$. Thus by the Minkowski convex body theorem, any lattice $\La$ must have a nonzero vector $\vv \in S$. Since $\La$ is not vertically short, we can assume that $\vv = (a, b)^T$ is not vertical, that is, $a \neq 0$. Further, we can assume that $a \geq 0$, otherwise we multiply by $-1$. Let $s_1 = \frac{b}{a}$. Then $(a, 0)^T \in h_{s_1}\La,$ so $h_{s_1} \La \in \Om$.

\end{proof}

\subsubsection{Proof of Theorem~\ref{theorem:main}}\label{subsec:proof:theorem:main} To prove Theorem~\ref{theorem:main}, we combine the above lemmas. Lemma~\ref{lemma:poincare} guarantees that all non-vertically short lattice horocycle orbits $\{h_s \La\}_{s \in \R}$ intersect $\Om$, and Lemma~\ref{lemma:roof} shows that if an $h_s$-orbit intersects $\Om$ once, it must intersect $\Om$ infinitely often (both forward and backward in time). To see that the set of intersection times is discrete, observe that the roof function $R(a,b)= \frac{1}{ab}$ is bounded below by $1$ on $\Om$, so visits must be spread out at least time $1$ apart. The calculation of the roof function $R$ and the return map $T$ are also given by Lemma~\ref{lemma:roof}. 

Finally, we address the remark following Theorem~\ref{theorem:main}. A direct calculation shows that if a lattice $\La$ has a vertical vector, then it is periodic under $h_s$ with period $t^2$, where $t$ is the length of the shortest vertical vector. If $\La$ is vertically short, this means that it is periodic under $h_s$ of period $\le 1$, so the $h_s$-orbit of $\La$ cannot intersect $\Om$. We will see in \S\ref{subsec:hyp} below that these orbits correspond to embedded closed horocycles, the family of which foliates the cusp of $\h^2/SL(2, \Z)$.\qed\medskip

\subsubsection{Observations on slopes}\label{sec:obs:slope} Given a unimodular lattice $\La$, let $$\{0 \le s_1 < s_2 < \ldots < s_N \ldots \}$$ denote the sequence of slopes of vectors in $\La \cap V_1$, where $$V_1 = \{ 0 < x \le 1, y>0\} \subset \R^2$$ is the vertical strip as in \S\ref{subsubsec:geomnumbers}. Then we see that $s_1$ is in fact the first hitting time of $\Omega$ of the positive orbit $h_s \La$, and the $s_n$ are the subsequent hitting times, since $\Om$ is the set of horizontally short lattices. The slopes of vectors decrease by $s$ under the horocycle flow $h_s$ (in particular, while $h_s$ does not preserve slopes, it preserves \emph{differences of slopes}), and the BCZ map records them when they become horizontal, while also keeping track of their horizontal component (the coordinate $a$) and the \emph{next} vector in the strip to become horizontal (the vector $(b, a^{-1})^T$). The differences in slopes are the return times $R$, that is, for $n \geq 1$, \begin{equation}\label{eq:slopegap} s_{n+1} -s_n = R(T^{n} (h_{s_1} \La)).\end{equation} This observation will be crucial for the proofs our results on the geometry of numbers, see \S\ref{sec:geomnum}.

\subsection{Hyperbolic geometry}\label{subsec:hyp} We recall that the there is a natural identification of the the upper-half plane $$\h^2 = \{z = x+iy: y >0\}.$$ with the space of unimodular lattices with a horizontal vector, via $$z \mapsto \frac{1}{\sqrt y} \left(\begin{array}{cc}1 & x \\0 & y\end{array}\right)\Z^2 = \La_{\frac{1}{\sqrt y}, \frac{x}{ \sqrt y}}.$$ That is, we identify $z$ with the unimodular lattice homothetic to the lattice generated by $1$ and $z$. This extends to an identification of all unimodular lattices with the unit-tangent bundle $T^1\h^2$, where the point $(z, i)$ (where $i$ denotes the upward pointing tangent vector) is identified to the lattice with a horizontal vector, and the point $(z, e^{i\theta} i)$ is identified to the lattice rotated by angle $\theta$. This identification is well-defined up to the action of $SL(2, \Z)$ by isometries on $\h^2$. In the standard fundamental domain for $SL(2, \Z)$, $$\left\{z=x+iy \in \h^2: |z| > 1, |x| < \frac 1 2\right\},$$ the set of horizontally short lattices can be identified with the strip $$C_1 : =\left\{z=x+iy \in \h^2: y > 1, |x| < \frac 1 2\right\},$$ together with their upward pointing tangent vectors, see Figure~\ref{fig:hyp}.  


\begin{figure}\caption{A hyperbolic picture of $\Om \equiv C_1$. The \textcolor{green}{green} circle is an $h_s$-orbit. Vertical lines are $g_t$-orbits and horizontal lines are $u_s$-orbits.\medskip}\label{fig:hyp}
\begin{tikzpicture}[scale=1.8]
   \draw(-2,0)--(2,0);
   \draw[dashed](180:1)node[below]{\tiny $-1$} arc (180:120:1);
   \draw[dashed](60:1) arc (60:0:1)node[below]{\tiny $1$};
   \draw(1/2, 3)--(60:1) arc (60:120:1)--(-1/2, 3);
   \draw[dashed] (1/2, 0)node[below]{\tiny $\frac 1 2$}--(60:1);
      \draw[dashed] (-1/2, 0)node[below]{\tiny -$\frac 1 2$}--(120:1);
\filldraw[fill=red!20!white] (-1/2, 3)--(-1/2,1)--(1/2,1)--(1/2, 3)--cycle;
\path(0, 2)node[blue]{$\Om$};
 \foreach \x in {-4, -3, -2, -1, 0, 1, 2, 3, 4}\draw[->](\x/10,1)--(\x/10,1.1);
  \foreach \x in {-4, -3, -2, -1, 0, 1, 2, 3, 4}\draw[->](\x/10,1.2)--(\x/10,1.3);
    \foreach \x in {-4, -3, -2, -1, 0, 1, 2, 3, 4}\draw[->](\x/10,1.4)--(\x/10,1.5);
        \foreach \x in {-4, -3, -2, -1, 0, 1, 2, 3, 4}\draw[->](\x/10,1.6)--(\x/10,1.7);
                \foreach \x in {-4, -3, -2, -1, 0, 1, 2, 3, 4}\draw[->](\x/10,1.8)--(\x/10,1.9);
                \foreach \x in {-4, -3, -2, -1, 0, 1, 2, 3, 4}\draw[->](\x/10,2.1)--(\x/10,2.2);
                \foreach \x in {-4, -3, -2, -1, 0, 1, 2, 3, 4}\draw[->](\x/10,2.3)--(\x/10,2.4);
                                \foreach \x in {-4, -3, -2, -1, 0, 1, 2, 3, 4}\draw[->](\x/10,2.5)--(\x/10,2.6);
                                                \foreach \x in {-4, -3, -2, -1, 0, 1, 2, 3, 4}\draw[->](\x/10,2.7)--(\x/10,2.8);
                                                \foreach \x in {-4, -3, -2, -1, 0, 1, 2, 3, 4}\draw[->](\x/10,2.9)--(\x/10,3);

\draw[->,green](0, 1.6)--(0, 1.7);
\draw[dashed, green]  (0,.8) circle (.8);

 \foreach \x in {15, 30, 45, 60, 75, 90, 105, 120, 135, 150, 165, 180, 195, 210, 225, 240, 255, 285, 300, 315, 330, 345, 360}\draw[yshift=0.8cm, ->, green](\x:.8)--(\x:.9);

\end{tikzpicture}
\end{figure}
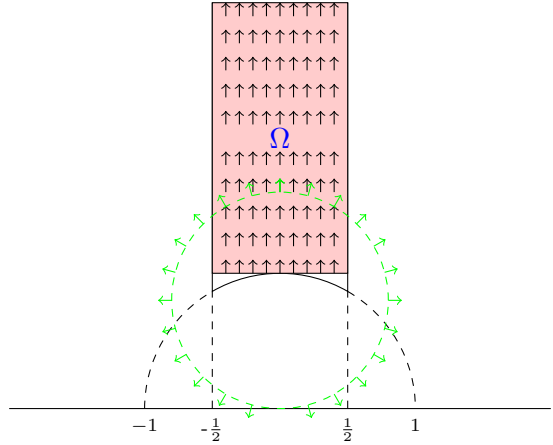

\subsubsection{Geodesics and horocycles} The orbits of the one-parameter subgroups $$A: = \left\{g_t = \left(\begin{array}{cc}e^{t/2} & 0 \\0 & e^{-t/2}\end{array}\right): t \in \R \right\}$$ and $$U: = \left\{u_s= \left(\begin{array}{cc}1 & s\\0 & 1\end{array}\right): t \in \R \right\},$$ the \emph{geodesic} and \emph{opposite horocycle} flows respectively, also have nice interpretations in terms of hyperbolic geometry. Under our choices, the action of $\{g_t: t \le 0\}$ moves the upward pointing tangent vectors vertically upward in $\h^2$, and the flow $\{u_s\}$ moves them horizontally. In particular, $\Om$ is preserved by the action of $\{g_t: t \le 0\}$ and $\{u_s\}$. In dynamical language, orbits of $h_s$ are leaves of the \emph{stable foliation} for $\{g_t: t \le 0\}$, and $\{u_s\}$ are leaves of the \emph{strong unstable foliation}. We are constructing a cross section of $h_s$ by taking a piece of the \emph{unstable foliation} $AU$. That is, $\Om \subset AU$. 

In the Euclidean picture, direct calculations show that $$g_t \La_{a, b} = \La_{e^{-t/2} a, e^{-t/2} b} \mbox{ and } u_s \La_{a, b} = \La_{a, b+sa^{-1}},$$ so geodesic orbits are radial lines and opposite horocyclic orbits are vertical lines, see Figure~\ref{fig:eucl}. We will use these observations in crucial ways in \S\ref{subsec:selfsim} below.

\begin{figure}\caption{\textcolor{red}{Geodesic} and \textcolor{blue}{opposite horocyclic} orbits in the Euclidean picture of $\Om$.\medskip}\label{fig:eucl}
\begin{tikzpicture}[scale=4]
\draw  (0, 0) -- (0,1.5);
\filldraw[fill=black!20!white] (1,0)--(1,1)--(0,1)--cycle;
\draw (0,0)--(1.5,0);

\draw[red](2/3, 1/3)--(1, 1/2)node[right]{\tiny $g_t$};
\draw[dashed, red](0, 0)--(2/3, 1/3);
\draw[blue](1/3, 2/3)--(1/3, 1)node[above]{\tiny $u_s$};

\path(2/3, 2/3)node{$\Om$};

\end{tikzpicture}
\end{figure}
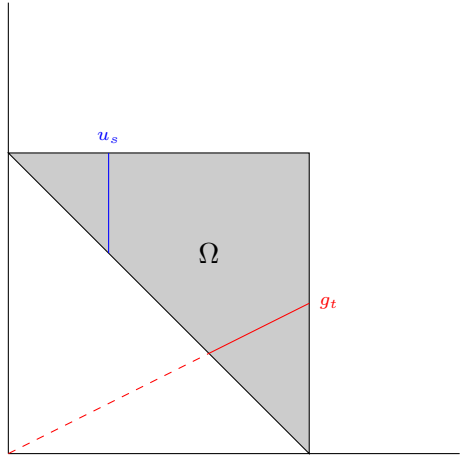

\subsubsection{Identifications}\label{subsubsec:ident} Note that in the fundamental domain $$\{z = x+iy: |z| \geq 1, |x| \le \frac 1 2 \}$$ for $SL(2, \Z)$ acting on $\h^2$, there is a natural identification of the vertical sides $x = \frac 1 2$ and $x = - \frac 1 2$ via the transformation $$z \mapsto z+1,$$ which is a fractional linear transformation associated to the unipotent matrix $$\left(\begin{array}{cc}1 & 1 \\0 & 1\end{array}\right).$$ In the Euclidean picture, we can also identify two boundaries of the region $\Om$ via the above unipotent matrix, namely the lines $\{x+y = 1: 0 < x \le 1\}$ and $\{y=1: 0 < x \le 1\}$. The resulting loop formed by the vertical segment $$\{x=1: 0<y \le 1\}$$ corresponds to the closed horocycle $$\{z=1+iy: 0 \le y \le 1\}$$ in the hyperbolic picture. Thus, the topology of $\Om$ is that of a sphere with one puncture (corresponding to the point at $\infty$ in the hyperbolic picture and the point $(0, 1)$ in the Euclidean picture) and one boundary component (corresponding to the loop described above).

\subsection{Self-similarity}\label{subsec:selfsim} The BCZ map has an extraordinary self-similarity property which can be seen naturally in both the hyperbolic and Euclidean pictures. Let $t>0$, and let $\Om_t$ denote the set of $t$-horizontally short lattices. Arguing as in Lemma~\ref{lem:short}, we have the identification \begin{equation}\label{eq:t:transversal} \Om_t : = \{ (a,b) \in (0, t]: a+b >t \}. \end{equation} Appropriately modifying Lemma~\ref{lemma:poincare}, we can define a return map (the $t$-\emph{BCZ map}) $$T_t: \Om_t \rightarrow \Om_t,$$ which captures the $h_s$-orbits of all lattices except those of $\frac{1}{t}$-vertically short lattices. Modifying the argument in Lemma~\ref{lemma:roof}, we have \begin{equation}\label{eq:tbcz} T_t(x, y) = \left(y, -x + \left \lfloor \frac{t+x}{y} \right\rfloor\right).\end{equation}
A direct calculation shows that the $t$-BCZ map is conjugate to the original ($1$-) BCZ map via the linear transformation $M_t: \Om \rightarrow \Om_t$ given by $$M_t(a,b) = (ta, tb),$$ that is \begin{equation}\label{eq:conj:discrete} T_t \circ  M_t = M_t \circ T.\end{equation} Now assume $t<1$ (if $t>1$, the roles of $t$ and $1$ below should be reversed). Then the set of $t$-horizontally short lattices can also be identified with the subset $\Om^{(t)} \subset \Om$ given by $$\Om^{(t)} : = \{ (x,y) \in \Om: x < t\}.$$ Let $T^{(t)}: \Om^{(t)} \rightarrow \Om^{(t)}$ be the first return map of the BCZ map $T$ to $\Om^{(t)}$. Thus, $T^{(t)}$ also represents the first return map of $h_s$ to the set of $t$-horizontally short lattices, and so $T^{(t)}$ and $T_t$ are conjugate. This conjugacy can be made explicit via the map $L_t:  \Om^{(t)} \rightarrow \Om_t$ given by $$ L_t(a,b) = (a, b-(1-t)),$$ that is, $$T_t \circ L_t = L_t \circ T^{(t)}.$$ In the hyperbolic picture, the set of $t$-horizontally short lattices can be identified with the subset $$C_t : = \left\{z=x+iy \in \h^2: y > t^{-2}, |x| < \frac 1 2\right\}.$$ Figures~\ref{fig:tbcz} and~\ref{fig:tbcz:hyp} show the Euclidean and hyperbolic pictures respectively. We can summarize this discussion with the following:
\begin{obs} $T$ is self-similar in the following sense: For any $0 < t<1$, the BCZ map $T$ is conjugate to its own first-return map $T^{(t)}$. \end{obs}
\medskip
\noindent Essentially, this self-similarity is a consequence of the following conjugation relation for $g_t$ and $h_s$ \begin{equation}\label{eq:conj} g_t h_s g_{-t} = h_{se^{-t}}, \end{equation} which in particular implies that the time-$1$ map $h_1$ and time-$s$ map $h_s$ are conjugate for any $s >0$ (via the map $g_{\log s}$).

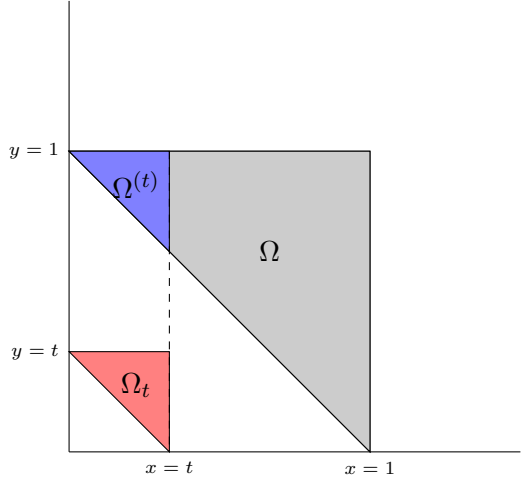
\begin{figure}\caption{\textcolor{red}{$\Omega_t$} and \textcolor{blue}{$\Omega^{(t)}$}. $\Omega$ can be identified with \textcolor{red}{$\Omega_t$}  via the scaling $M_t$, and \textcolor{red}{$\Omega_t$}  and \textcolor{blue}{$\Omega^{(t)}$} via the vertical translation $L_t$.\medskip}\label{fig:tbcz}
\begin{tikzpicture}[scale=4]
\draw  (0, 0) -- (0,1.5);
\filldraw[fill=black!20!white] (1,0)node[below]{\tiny $x=1$}--(1,1)--(0,1)node[left]{\tiny $y=1$}--cycle;
\draw (0,0)--(1.5,0);

\draw[blue](1/3, 2/3)--(1/3, 1);

\filldraw[fill=red!50!white] (1/3,0)--(1/3,1/3)--(0,1/3)node[left]{\tiny $y=t$}--cycle;

\filldraw[fill=black!20!white] (1,0)--(1,1)--(0,1)--cycle;

\path(2/3, 2/3)node{$\Om$};

\filldraw[fill=blue!50!white] (0,1)--(1/3,2/3)--(1/3,1)--cycle;
\draw[dashed](1/3, 1)--(1/3, 0)node[below]{\tiny $x =t$};

\path(2/9, 2/9)node{$\Om_t$};

\path(2/9, 8/9)node{$\Om^{(t)}$};

\end{tikzpicture}
\end{figure}

\begin{figure}\caption{A hyperbolic picture of \textcolor{red}{$\Om_t $} $\subset \Om$, that is, $C_t \subset C_1$.\medskip}\label{fig:tbcz:hyp}
\begin{tikzpicture}[scale=1.8]
   \draw(-2,0)--(2,0);
   \draw[dashed](180:1)node[below]{\tiny $-1$} arc (180:120:1);
   \draw[dashed](60:1) arc (60:0:1)node[below]{\tiny $1$};
   \draw(1/2, 3)--(60:1) arc (60:120:1)--(-1/2, 3);
   \draw[dashed] (1/2, 0)node[below]{\tiny $\frac 1 2$}--(60:1);
      \draw[dashed] (-1/2, 0)node[below]{\tiny -$\frac 1 2$}--(120:1);
\filldraw[fill=black!20!white] (-1/2, 3)--(-1/2,1)--(1/2,1)--(1/2, 3)--cycle;
\path(0, 1.2)node{$\Om$};

\filldraw[fill=red!50!white] (-1/2, 3)--(-1/2,1.5)--(1/2,1.5)node[right]{\tiny $y = t^{-2}$}--(1/2, 3)--cycle;
\path(0, 2.2)node[red]{$\Om_t$};

\end{tikzpicture}
\end{figure}
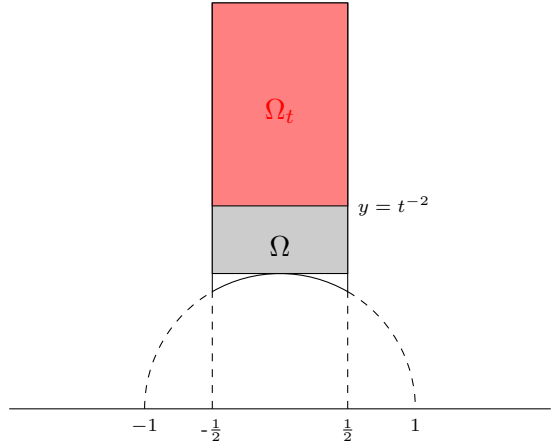

\section{Ergodic properties of the BCZ map}\label{sec:ergbcz} \noindent In this section, we prove Theorem~\ref{theorem:BCZ:ergodic}, using Theorem~\ref{theorem:main} and well-known properties of suspension flows. We first recall these properties, which can be found, e.g., in~\cite{Nadkarni}.

\subsection{Suspension Flows}\label{subsec:suspension}  Let $(X, \mu)$ be a measure space, and $T: X \rightarrow X$ a measure-preserving bijection. Let $R: X \rightarrow \R^+$ be in $L^1(X, \mu)$. The \emph{suspension flow} over $T$ with roof function $R$ is defined on the space \begin{equation}\label{eq:suspension:space} X^R : = \{ (x, t): x \in X, t \in [0, R(x))\}/\sim \end{equation} where $(x, R(x)) \sim (T(x), 0)$. The \emph{suspension flow} $\phi^{R, T}$ is given by $$\phi^{R, T}_{s}(x, t) = (x, s+t).$$ The measure $d\nu_{R, T} = \frac{1}{\|R\|_1} d\mu dt$ is a $\phi^{R,T}$-invariant probability measure on $X^R$. It is ergodic if and only if $\mu$ is $T$-ergodic.The natural dual construction to suspension flow is the construction of a \emph{first return map for a flow}. Given a measure-preserving flow $\phi: (Y, \nu) \rightarrow (Y, \nu)$, and a subset $X \subset Y$ so that for almost all $y \in Y$, $\{t \in \R: \phi^t y \in X \}$ is discrete, we define the \emph{first return map} $T: X \rightarrow X$ by $T(x) = \phi^{R(x)} (x)$, where $$R(x) = \min\{t >0: \phi^t(x) \in X\}$$ is the \emph{first return map}. There is a natural (possibly infinite) $T$-invariant measure $\mu$ on $X$ so that $R \in L^1(X, \mu)$. and the suspension flow $\phi^{R, T}: (X^R, \nu_{R, T}) \rightarrow (X^R, \nu_{R, T})$ is naturally isomorphic to the original flow $\phi$. There is also a map between the sets of invariant measures for the flow $\phi^{R,T}$ and the map $T$, whose properties are summarized in the following:

\begin{lemma}\label{lemma:suspension:properties} Let $(X, \mu)$ be a measure space, and let $T:(X, \mu) \rightarrow (X, \mu)$ be a $\mu$-preserving bijection. Let $R$ be a positive function in $L^1(X, \mu)$. The map $\eta \longmapsto \eta_{R, T}$, with $d\eta_{R, T} = \frac{1}{\|R\|_{1, \eta}} d\eta dt$ is a bijection between the set of $T$-invariant measures $\eta$ on $X$ so that $R \in L^1(X, \eta)$, and the set of $\phi^{R, T}$ invariant probability measures on $X^R$ (and thus also between the ergodic invariant measures). Moreover, we have \emph{Abramov's formula} relating the entropy of the flow $\phi^{R, T}$ (that is, of its time $1$-map) to the entropy of the map $T$:\begin{equation}\label{eq:entropy} h_{\eta^{R,T}}(\phi^{R, T}_1) = \frac{h_{\eta}(T)}{\|R\|_{1, \eta}}.\end{equation}
\end{lemma}

\subsection{Ergodicity}\label{subsec:ergodic} In this section, we prove the ergodicity and measure-classification parts of Theorem~\ref{theorem:BCZ:ergodic}.  The ergodicity of the $BCZ$ map with respect to $dm = 2dadb$ follows from 
\begin{description}

\item[Ergodicity] of the horocycle flow with respect to Haar measure $\mu_2$ on $X_2$.
\medskip 

\item[Suspension]  The measure $2 da db ds$ on $X_2$ (viewing it as the suspension space over $\Om$) is an absolutely continuous invariant measure for $\{h_s\}$. 
\medskip
\item[Measure classification] By Dani's measure classification~\cite{Dani}, this must be (a constant multiple of) $\mu_2$.
\medskip
\end{description}
To show uniqueness, suppose $\nu$ was another ergodic $T$ invariant measure on $\Om$. Then it must be supported on a periodic orbit for $T$, since the measure $d\nu ds$ on $X_2$ is a non-Haar ergodic invariant measure for $\{h_s\}$ and thus, by Dani's measure classification, must be supported on a periodic orbit for $\{h_s\}$.\qed\medskip

\subsection{Entropy}\label{subsec:entropy} We now prove the fact that the entropy $h_m(T)$ of $T$ with respect to the Lebesgue probability measure $m$ is $0$, completing the proof of Theorem~\ref{theorem:BCZ:ergodic}. We use Lemma~\ref{lemma:suspension:properties}), which states that the entropy $h_m(T)$ of the BCZ map with respect to $m$ is proportional to the entropy of the horocycle flow $h_{\mu_2}(h_1)$ with respect to Haar measure. To show that $h_{\mu_2}(h_1)=0$, we record a fact which seems to be well-known but does not have a complete proof in the literature as far as we are aware:

\begin{Theorem}\label{theorem:horocycle:entropy} Let $\Gamma \subset SL(2, \R)$ be a lattice. Let $\mu_2$ denote the finite measure on $SL(2, \R)/\Gamma$ induced by Haar measure. Then $$h_{\mu_2}(h_1) =0.$$ In fact, the topological entropy $$h_{top}(h_1) =0.$$\end{Theorem}
\begin{proof} When $\Gamma$ is a \emph{uniform} lattice, that is, the quotient $\h^2/\Gamma$ is compact, the second assertion is a theorem of Gur\v{e}vic~\cite{Gurevic}, and together with the standard variational principle (see, e.g., Walters~\cite{Walters}), which states that the topological entropy is the supremum of the measure-theoretic entropies. When $\Gamma$ is non-uniform, we can use a version of the variational principle for locally compact spaces due to Handel-Kitchens~\cite{HK}, to again reduce the measure-theoretic statement to a topological statement. Since the horocycle flow is $C^{\infty}$, we can apply a theorem of Bowen~\cite{Bowen} to conclude that $$h_{top} (h_1) < \infty.$$ On the other hand, since $h_s$ is conjugate to $h_1$ for any $s>0$, we have $$h_{top}(h_1) = h_{top}(h_s) = sh_{top}(h_1),$$ which implies that $h_{top}(h_1) =0$.
\end{proof}

\section{Structure of Periodic Orbits}\label{sec:period}

In this section we prove Theorems~\ref{theorem:dense:periodic} and~\ref{theorem:structure:periodic} and describe in detail the relationship between periodic orbits for $T$ and $h_s$. We first record (in \S\ref{subsec:fareyperiod}) our geometric proof of the key observation (\ref{eq:farey:bcz}) that certain periodic orbits of the BCZ map parameterize Farey fractions.

\subsection{Farey fractions and lattices}\label{subsec:fareyperiod} In this section, we give a geometric proof the following
\begin{lemma}\label{lem:farey:key} Let $Q$ be a positive integer, and let $$\F(Q) =\{ \frac{0}{1} = \gamma_1 < \gamma_2 = \frac{1}{Q} < \ldots < \gamma_i = \frac{p_i}{q_i} < \ldots < \gamma_{N} = \frac 1 1\},$$ $N = N(Q) = \sum_{q \le Q} \varphi(q)$ denote the Farey sequence. Then (\ref{eq:farey:bcz}) holds, that is, $$T^i\left(\frac{1}{Q}, 1\right) = \left(\frac{q_i}{Q}, \frac{q_{i+1}}{Q}\right).$$\end{lemma}
\medskip
\noindent\textbf{Remark:} Here, $i$ is interepreted cyclically (i.e., in $\Z/N\Z$). In particular, this is a \emph{periodic orbit} for $T$ (and for $h_s$). Note also that the return time gives (normalized) gaps between the Farey fractions, that is, 
$$R \circ T^i\left(\frac{1}{Q}, 1\right) = \frac{Q^2}{q_i q_{i+1}} = Q^2 (\gamma_{i+1} - \gamma_i).$$

\begin{proof}
Geometrically, the sequence $\F(Q)$ correspond to the slopes of primitive integer vectors $\left(\begin{array}{c}q_i \\p_i\end{array}\right)$ in the (closed) triangle $T_Q$ with vertices at $(0,0)$, $(Q, 0)$, and $(Q, Q)$. Now note that since $\Z^2$ contains a vertical vector, it is \emph{periodic} under $h_s$, with period 1. Using the conjugation relation (\ref{eq:conj}), we have that the lattice $g_t \Z^2$ has period $e^{-t}$ under $h_s$, since
$$h_{e^{-t}} g_t \Z^2 = g_{t} h_{1} \Z^2 = g_t \Z^2.$$  
Note that $g_t$ scales slopes (and their differences) by $e^{-t}$. For $Q \in \N$, setting $t_Q = -2\ln Q$, consider the lattice
$$g_{t_Q} \Z^2 = \left(\begin{array}{cc}Q^{-1} & 0 \\0 & Q\end{array}\right)\Z^2$$
Note that we can also write this as
$$\left(\begin{array}{cc}Q^{-1} & 0 \\0 & Q\end{array}\right)\Z^2 = \left(\begin{array}{cc}Q^{-1} & 1\\0 & Q\end{array}\right) \left(\begin{array}{cc}1 & -Q \\0 & 1\end{array}\right)\Z^2 = \left(\begin{array}{cc}Q^{-1} & 1 \\0 & Q\end{array}\right)\Z^2$$
This lattice corresponds to the point $(Q^{-1}, 1)$ in the Farey triangle $\Om$. By our earlier observation, this has period $Q^2$ under $h_s$. We are interested in the behavior of the orbit of this point under $T$. Recall that the BCZ map only `sees' primitive vectors with horizontal component less than 1 that is, in the strip $V= V_1$, since $h_s$ does not change the horizontal component of vectors. Originally we were interested in Farey fractions of level $Q$, that is, the slopes of (primitive) integer vectors in the region $T_Q$. Applying the matrix $g_{t_Q}$, the region $T_Q$ is transformed into a subset $V$. Thus, the vectors $\left(\begin{array}{c}q_i \\p_i\end{array}\right) \in T_Q$ correspond to vectors $\left(\begin{array}{c}Q^{-1}q_i \\Q p_i\end{array}\right) \in g_{t_Q} T_Q \subset V$. Therefore, they are seen by the BCZ map. Precisely, $g_{t_Q} T_Q$ is the triangle with vertices at $(0,0)$, $(1,0)$, and $(1, Q^2)$. See Figure~\ref{fig:scaledtriangle:image}.

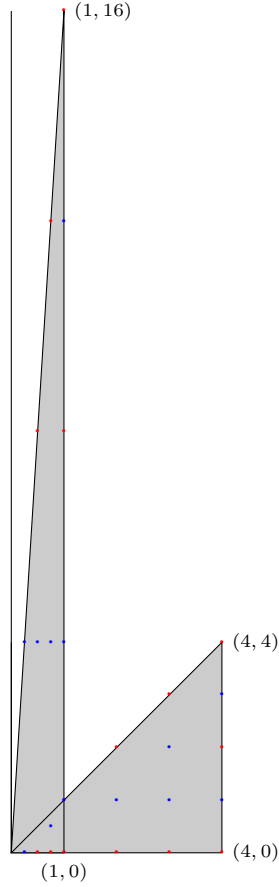
\begin{figure}\caption{The triangles $T_4$ and $g_{t_4}T_4$, with lattice points. Primitive lattice points in \textcolor{blue}{blue}, others are in \textcolor{red}{red} .}\label{fig:scaledtriangle:image}
\begin{tikzpicture}[scale=.7]
\filldraw[fill=black!20!white] (0,0)--(4,4)--(4,0)--cycle;
\draw (0,0)--(0,4);
\path (4, 0) node[right]{\tiny $(4,0)$};
\path (4, 4) node[right]{\tiny $(4,4)$};
\path (4, 4) node{\textcolor{red}{$\cdot$}};
\path (3, 3) node{\textcolor{red}{$\cdot$}};
\path (2, 2) node{\textcolor{red}{$\cdot$}};
\path (1, 1) node{\textcolor{blue}{$\cdot$}};
\path (4, 0) node{\textcolor{red}{$\cdot$}};
\path (3, 0) node{\textcolor{red}{$\cdot$}};
\path (2, 0) node{\textcolor{red}{$\cdot$}};
\path (1, 0) node{\textcolor{blue}{$\cdot$}};
\path (4, 3) node{\textcolor{blue}{$\cdot$}};
\path (4, 2) node{\textcolor{red}{$\cdot$}};
\path (4, 1) node{\textcolor{blue}{$\cdot$}};
\path (3, 2) node{\textcolor{blue}{$\cdot$}};
\path (3, 1) node{\textcolor{blue}{$\cdot$}};
\path (2, 1) node{\textcolor{blue}{$\cdot$}};
\filldraw[fill=black!20!white] (0,0)--(1,16)--(1,0)--cycle;
\draw (0,0)--(0,16);
\path (1, 0) node[below]{\tiny $(1,0)$};
\path (1, 16) node[right]{\tiny $(1,16)$};
\path (1, 16) node{\textcolor{red}{$\cdot$}};
\path (3/4, 12) node{\textcolor{red}{$\cdot$}};
\path (1/2, 8) node{\textcolor{red}{$\cdot$}};
\path (1/4, 4) node{\textcolor{blue}{$\cdot$}};
\path (1, 0) node{\textcolor{red}{$\cdot$}};
\path (3/4, 0) node{\textcolor{red}{$\cdot$}};
\path (2/4, 0) node{\textcolor{red}{$\cdot$}};
\path (1/4, 0) node{\textcolor{blue}{$\cdot$}};
\path (1, 12) node{\textcolor{blue}{$\cdot$}};
\path (1, 8) node{\textcolor{red}{$\cdot$}};
\path (1, 4) node{\textcolor{blue}{$\cdot$}};
\path (3/4, 1/2) node{\textcolor{blue}{$\cdot$}};
\path (3/4, 4) node{\textcolor{blue}{$\cdot$}};
\path (1/2, 4) node{\textcolor{blue}{$\cdot$}};
\draw (0,0)--(1,1);
\path (1,1)node[blue]{$\cdot$};
\end{tikzpicture}
\end{figure}

\medskip
\noindent As noted above, the (primitive) vectors $\left(\begin{array}{c}Q^{-1} q_i \\Q p_i\end{array}\right)  \in g_{t_Q} \Z^2$ will be seen by the BCZ map in the order of their slopes (that is, at time $Q^2 \gamma_i$), and the time in between the $i^{th}$ and $(i+1)^{st}$ vectors will precisely be the difference in their slopes 
$$Q^2(\gamma_{i+1} - \gamma_i).$$

\noindent At time $0$ we are seeing the horizontal vector $\left(\begin{array}{c}Q^{-1} \\0\end{array}\right)$ (corresponding to $\gamma_0$), and the next vector to get short is $\left(\begin{array}{c}1 \\Q\end{array}\right)$ (corresponding to $\gamma_1$), which will become horizontal in time 
$$R(Q^{-1}, 1) = Q = Q^2 \left(\frac 1 Q - 0\right) = Q^2 (\gamma_1 - \gamma_0).$$
More generally, we have that $T^i(Q^{-1}, 1)$ corresponds to the lattice $h_{Q^2 \gamma_i} g_{t_Q} \Z^2$, and noting 
$$h_{Q^2\gamma_i} g_{t_Q}\left(\begin{array}{cc}q_i & q_{i+1} \\p_i & p_{i+1}\end{array}\right) = \left(\begin{array}{cc}Q^{-1}q_i &Q^{-1} q_{i+1} \\0 & Q q_i^{-1}\end{array}\right).$$
we obtain, as desired
$$T^{i} (Q^{-1}, 1) = \left(\frac{q_i}{Q}, \frac{q_{i+1}}{Q}\right),$$
and
$$R(T^{i} (Q^{-1}, 1)) = Q^2(\gamma_{i+1} - \gamma_i).$$
In particular, the point $(Q^{-1}, 1)$ is $T$-periodic with period $N(Q)$. Also note that 
$$\sum_{i=0}^{N(Q) -1} R(T^{i} (Q^{-1}, 1)) = Q^2,$$
and thus, the sum of the return times over the discrete periodic orbit correspond, as they must, to the continuous period. \end{proof}

\subsection{Proof of Theorem~\ref{theorem:dense:periodic}}

We first note that any periodic point $(a,b) \in \Om$ for $T$ corresponds to a periodic point $p_{a,b} \Z^2 \in \tilde{Y}$ for $h_s$ on $X_2$ and vice-versa. Precisely, suppose $(a, b) \in \Om$ is periodic under $T$, and $n = P(a,b)>0$ is the minimal period, so $$T^n (a,b) =(a,b)$$ and $T^j(a,b) \neq (a,b)$ for $j < n$. Then $$h_{s(a,b)} p_{a,b} \Z^2 = p_{a,b} \Z^2,$$ where $$s(a,b) = \sum_{i=0}^{n-1} R(T^i(a,b)).$$ Note that the points $(a,a) \in \Om\backslash \Om$ have period $P(a,a) =1$, and that a direct calculation shows that $s(a,a) = R(a,a) = \frac{1}{x^2}$.

Periodic orbits for $h_s$ occur in natural families, due the conjugation relation (\ref{eq:conj}). In particular, if $\La$ is periodic with period  $s_0$, \begin{equation}\label{eq:conj:period} h_{s_0 e^{-t}} g_t \La = g_t h_{s_0} \La = g_t \La,\end{equation} so the flow period of $g_t g \Z^2$ is $s_0 e^{-t}$. We have $$g_t \La_{a,b} = \La_{e^{t/2} a, e^{t/2} b},$$ so we see that if $(a, b)$ is periodic for $T$, so is the entire line segment $$\left\{(ta, tb), 1 \geq t > \frac{1}{a+b}\right\}.$$

Thus, to analyze which points have periodic orbits, it suffices to consider points of the form $(1, b)$ or $(a, 1)$. Since $T(a, 1) = (1, 1-a)$, we consider points of the first kind, or, equivalently, lattices of the form $\La_{1, b}$. This is a periodic point under $h_s$ if and only if there exists an $s_0$ so that $h_{s_0} \La_{1, b} = \La_{1,b}$, or, equivalently, $$ p_{1, b}^{-1} h_{s_0} p_{1,b} \in SL(2, \Z).$$ A direct calculation shows that this is impossible if $y$ is irrational, and if $y = \frac{k}{l}$ the minimal such $s_0 = l^2$, and we have \begin{equation}\label{eq:period:matrix}  p_{1, b}^{-1} h_{s_0} p_{1,b} =  \left(\begin{array}{cc}1+kl & k^2 \\-l^2 & 1-kl\end{array}\right).\end{equation} This proves Theorem~\ref{theorem:dense:periodic}, since we have shown that any point with rational slope is periodic.\qed\medskip

\subsection{Proof of Theorem~\ref{theorem:structure:periodic}}

As above, we consider the point $(1, b)$. If $b = \frac{k}{l}$, $k \le  l \in \N$ relatively prime, then the period under $h_s$ is $l^2$.  The line segment associated to this point is $$\left\{\left(t, t \frac k l\right) : t \in \left( \frac{l}{l + k}, 1\right]\right\}.$$ Applying (\ref{eq:conj}) to $p_{t, t\frac{k}{l}} \Z^2$, we can calculate the period $$s\left(t, t\frac{k}{l}\right)= \frac{l^2}{t^2}.$$ To calculate the discrete period, we analyze the scalings of the Farey periodic orbits $$\left\{T^i\left(1, \frac{1}{Q}\right) \right\}_{i=0}^{N(Q) -1}$$ using the following proposition, a special case of Theorem~\ref{theorem:structure:periodic} which we will use to prove the general result:

\begin{Prop}\label{prop:farey:period} For $t \in \left(\frac{Q}{Q+1}, 1\right]$, $$P\left(t, \frac{t}{Q}\right)= N(Q).$$
\end{Prop}

\medskip

This will follow from showing that $T$ in fact is \emph{linear} along the orbit of the segment $\left\{\left(t, \frac {t} {Q} \right): t \in \left(\frac{Q}{Q+1}, 1\right]\right\}$, that is, the segment does not `break up' into pieces. Precisely, we have:

\begin{lemma}\label{lemma:farey:period:index} For $t \in \left(\frac{Q}{Q+1}, 1\right]$, $1 \le i \le N(Q)$, $$\kappa\left(T^i \left(t, \frac t Q\right)\right) = \kappa\left(T^i \left(1, \frac 1 Q\right)\right).$$
\end{lemma}
\begin{proof} Since $T^i \left(1, \frac 1 Q\right) = \left(\frac{q_i}{Q}, \frac{q_{i+1}}{Q}\right)$, we need to show, for $t \in \left(\frac{Q}{Q+1}, 1\right]$
$$\kappa\left(\frac{q_i}{Q}, \frac{q_{i+1}}{Q}\right) = \kappa\left(\frac{tq_i}{Q}, \frac{tq_{i+1}}{Q}\right).$$ We have \begin{eqnarray}\label{eqnarray:index} \kappa\left(\frac{tq_i}{Q}, \frac{tq_{i+1}}{Q}\right) &=& \left\lfloor \frac{ \frac{Q}{t} + q_i}{q_{i+1}} \right\rfloor \\ \nonumber &\le& \frac{ \frac{Q}{t} + q_i}{q_{i+1}}\\ \nonumber &<& \frac{Q+1 + q_i}{q_{i+1}} \\ \nonumber & \le & \frac{q_{i+1} + q_{i+2} + q_i}{q_{i+1}} \\ \nonumber &=& 1 + \left \lfloor \frac{Q+q_i}{q_{i+1}}\right\rfloor \end{eqnarray}
where in the last line we are using the identity $$\frac{q_i + q_{i+2}}{q_{i+1}} = \left \lfloor \frac{Q+q_i}{q_{i+1}}\right\rfloor.$$ On the other hand, $$\frac{ \frac{Q}{a} + q_i}{q_{i+1}} \geq \frac{ Q + q_i}{q_{i+1}} \geq \left \lfloor \frac{ Q + q_i}{q_{i+1}} \right \rfloor.$$ Thus, $$ \left \lfloor \frac{ Q + q_i}{q_{i+1}} \right \rfloor \leq \kappa\left(\frac{tq_i}{Q}, \frac{tq_{i+1}}{Q}\right) < 1 + \left \lfloor \frac{Q+q_i}{q_{i+1}}\right\rfloor,$$ so $$\kappa\left(\frac{tq_i}{Q}, \frac{tq_{i+1}}{Q}\right) = \left \lfloor \frac{Q+q_i}{q_{i+1}}\right\rfloor = \kappa\left(\frac{q_i}{Q}, \frac{q_{i+1}}{Q}\right)$$as desired.
\end{proof}
\medskip
\noindent Now Proposition~\ref{prop:farey:period} follows from $P\left(1, \frac 1 Q\right) = N(Q)$.\qed\medskip

\noindent To prove Theorem~\ref{theorem:structure:periodic}, we observe that the point $\left(1, \frac{k}{l}\right)$ is contained in the periodic orbit corresponding to the point $\left(\frac{1}{l}, 1\right)$, since we can find consecutive Farey fractions $\gamma_i = \frac{a_i}{l}$ and $\gamma_{i+1}= \frac{a_{i+1}}{k}$ in $\F(l)$ (we are assuming that $\mbox{gcd}(k,l) = 1$). By Proposition~\ref{prop:farey:period}, for $a \in (\frac{l}{l+1}, 1]$, $$P\left(a, a\frac{k}{l}\right) = N(l).$$
If we put $a = \frac{l}{l+1}$, we have the point $\left(\frac{l}{l+1}, \frac{k}{l+1}\right)$, which is contained in the periodic orbit of the point $\left(\frac{1}{l+1} , 1\right)$, by similar reasoning as above.  Applying Proposition~\ref{prop:farey:period} to this point, we have that for $b \in \left(\frac{1}{l+1}, 1\right]$, $$P\left(\frac{b l}{l+1}, \frac{bk}{l+1}\right) = N(l+1).$$ Rewriting, we have that for $a \in \left(\frac{l}{l+2},\frac{l}{l+1}\right]$, $$P\left(t, t\frac{k}{l}\right) = N(l+1).$$ Continuing to reason in this fashion, we have that for $1 \le r \le k$, $t \in \left(\frac{l}{l+r}, \frac {l}{l+r-1}\right]$, $$P\left(t, t\frac{k}{l}\right) = N(l+r-1).$$ Finally, we need to calculate $A_{n} \left(t, t\frac{k}{l}\right),$ where $n = P\left(t, t\frac{k}{l}\right).$ Letting $s_0 = s\left(a, a\frac{k}{l}\right) = \frac{l^2}{a^2},$ we have, by a similar calculation to (\ref{eq:period:matrix}) that $$p_{t, t\frac{k}{l}}^{-1} h_{s_0} p_{t, t \frac k l} =  \left(\begin{array}{cc}1+kl & k^2 \\-l^2 & 1-kl\end{array}\right).$$ On the other hand, (\ref{eq:cocycle}) can be re-written as the conjugation $$p_{T(a,b)}^{-1} h_{R(a,b)} p_{a,b} = A(a,b)^T.$$ Iterating, we obtain, for any $m >0, (a,b) \in \Om$, $$p_{T^m(a,b)}^{-1} h_{R(a,b)} p_{a,b} = A_{m}(a,b)^T.$$ Applying this to $(a,b) = \left(t, t\frac k l\right)$ and $m = n$, we have $$\left(A_n(a,b)^-1\right)^T  = \left(\begin{array}{cc}1+kl & k^2 \\-l^2 & 1-kl\end{array}\right),$$ since $T^n\left(t, t\frac k l\right) = \left(t, t\frac k l\right)$. Inverting and taking transposes, we obtain, as desired  $$A_{P\left(t, t\frac{k}{l}\right)}\left(t, t\frac{k}{l}\right) = \left(\begin{array}{cc}1-kl & l^2 \\-k^2 & 1+kl\end{array}\right),$$ completing the proof.\qed\medskip

 \subsection{Segments and the shearing matrix}\label{subsec:segments} Note that the matrix  $$A_{P\left(a, a\frac{k}{l}\right)}\left(a, a\frac{k}{l}\right) = \left(\begin{array}{cc}1-kl & l^2 \\-k^2 & 1+kl\end{array}\right)$$ is parabolic, and fixes the segment $\left\{\left(a, a\frac{k}{l}\right): a \in \left(\frac{l}{l+k}, 1\right]\right\}.$ In fact, it \emph{shears} along this segment in the following fashion. Note that $$\frac{1}{k^2 + l^2} \left(\begin{array}{cc}l & k \\-k & l\end{array}\right)\left(\begin{array}{cc}1-kl & l^2 \\-k^2 & 1+kl\end{array}\right)\left(\begin{array}{cc}l & -k \\k & l\end{array}\right) = \left(\begin{array}{cc}1 & l^2 +k^2 \\0 & 1\end{array}\right).$$ Thus for points $(a,b)$ sufficiently near the segment, they will be sheared along the segment under the appropriate power of $T$ (which may vary depending on the period of the subsegment $(a,b)$ is close to). As it moves along the segment, the period will change. This is illustrated in Figure~\ref{fig:periodic:segment}, which shows the periodic segment $\left\{\left(a, a\frac 2 3\right): a \in \left(\frac 3 5, 1\right]\right\}$, which breaks up into two pieces, one of period $5$ and one of period $6$. The shearing matrix for this line is given by $$\left(\begin{array}{cc}-5 & 9 \\-4 & 7\end{array}\right),$$ which is conjugate to $$\left(\begin{array}{cc}1 & 13 \\0 & 1\end{array}\right).$$
 
\begin{figure}

\caption{The periodic segment $\left\{\left(a, a\frac 2 3\right): a \in \left(\frac 3 5, 1\right]\right\}$. The segment \textcolor{blue}{$a \in \left (\frac 3 4, 1\right)$} is in \textcolor{blue}{blue} and the segment \textcolor{red}{$a \in \left(\frac 3 5, \frac 3 4\right]$} is in \textcolor{red}{red}. The labels next to the segments correspond to the power of $T$. The \textcolor{red}{red} segment has period \textcolor{red}{$N(4) =6$}, and the \textcolor{blue}{blue} segment has period \textcolor{blue}{$N(3) =4$}.} \label{fig:periodic:segment}
\vspace{.2in}

\begin{tikzpicture}[scale=6]

\filldraw[fill=black!10!white] (1,0)node[right]{\tiny$(1,0)$}--(1,1)node[right]{\tiny$(1,1)$}--(0,1)node[above]{\tiny$(0,1)$}--cycle;

\draw[blue] (1, 2/3)-- (3/4, 1/2); 
\path (7/8, 7/12)node[right, blue]{\tiny $0$};

\draw[dashed] (1/3, 2/3)-- (1, 1);

\draw[dashed] (1/2, 1/2)-- (1, 2/3);

\draw[dashed] (3/5, 2/5)-- (1, 1/2);

\draw[dashed] (2/3, 1/3)-- (1, 2/5);

\draw[dashed] (5/7, 2/7)-- (1, 2/6);
\draw[dashed] (3/4, 1/4)-- (1, 2/7);
\draw[dashed] (7/9, 2/9)-- (1, 2/8);
\draw[dashed] (8/10, 2/10)-- (1, 2/9);

\draw[cm={0,-1,1,3,(0,0)},blue]	(1,2/3)--(3/4, 1/2);
\path (7/12, 7/8)node[below, blue]{\tiny $1$};

\draw (1,0)--(0,0)node[below]{\tiny$(0,0)$}--(0,1);

\draw[blue] (3/4, 1/4)-- (1, 1/3); 
\path (7/8, 7/24)node[right, blue]{\tiny $2$};

\draw[blue] (1/3, 3/4)--(1/4, 1);

\path (7/24, 7/8)node[below right, blue]{\tiny $3$};

\draw[red] (3/4,1/2)--(3/5,2/5);
\path (27/40, 9/20)node[below, red]{\tiny $0$};

\draw[red] (1/2,3/4)--(2/5,3/5);
 \path (9/20, 27/40)node[below, red]{\tiny $1$};
 
 \draw[red] (3/4, 1)--(3/5,4/5);
 \path (27/40, 9/10)node[below, red]{\tiny $2$};
 
  \draw[red] (1,1/4)--(4/5,1/5);
 \path (9/10, 9/40)node[below, red]{\tiny $3$};

  \draw[red] (1/4,1)--(1/5, 4/5);
 \path (9/40, 9/10)node[right, red]{\tiny $4$};
 
   \draw[red] (1,3/4)--(4/5,3/5);
 \path (9/10, 27/40)node[below, red]{\tiny $5$};


 \end{tikzpicture}
 \end{figure}
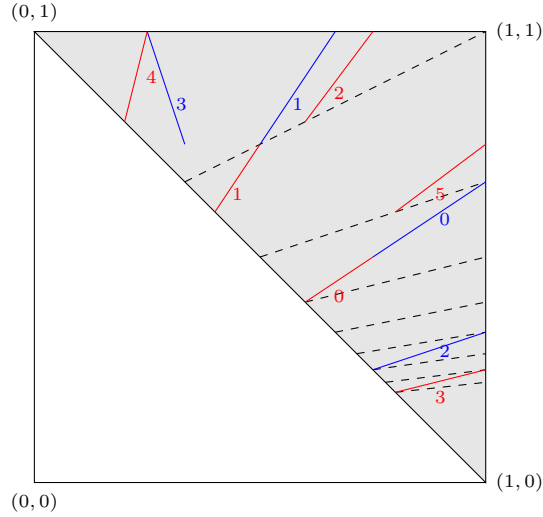

\subsection{Hierarchies of periodic orbits}\label{subsec:hierarchies:periodic} Theorem~\ref{theorem:structure:periodic} has a nice geometric explanation, which we describe informally: If we start with the points  $(a,a) \in \Om$ of period $1$ under $T$, and move down the line, we hit $\left(\frac{1}{2}, \frac{1}{2}\right)$, which is not in $\Om$. However, we can identify $\La_{\frac{1}{2},\frac{1}{2}}$ with $\La_{\frac{1}{2}, 1}$, since, in general, we have $$\La_{a, 1-a} = \La_{a,1}.$$ The orbit of $\left(\frac{1}{2}, 1\right)$ has period $N(2) = 2$, and continuing to move down the line, Proposition~\ref{prop:farey:period} shows that for $a \in (2/3, 1]$, $$P\left(\frac{a}{2}, a\right) = N(2).$$ At $a = \frac{2}{3}$, we identify $\left(\frac 1 3, \frac 2 3\right)$ with the point $\left(\frac{1}{3}, 1\right)$, which has period $N(3) = 5.$ Continuing this process, we obtain all the Farey periodic orbits $\left(\frac{1}{Q}, 1\right)$ and the scaling segments $$\left\{\left(\frac{t}{Q}, t\right): t \in \left(\frac{Q}{Q+1}, 1\right]\right\}.$$ The set of $T$-periodic points in $\Om$ is the union $$ \bigcup_{Q \in \N} \bigcup_{1 \le i \le N(Q)} \left\{T^i \left(\frac{t}{Q}, t\right): t \in \left(\frac{Q}{Q+1}, 1\right]\right\}.$$

\subsection{Hyperbolic Geometry of Periodic Orbits}

What we have done in $X_2$ is to push the cuspidal horocycle correpsonding to $\La_{a,a}$, $a >1$ down using $g_t$.  For $a>1$, the lattice $\La_{a, a}$ is vertically short, and the $h_s$ trajectory is an embedded loop in $X_2$ which does not intersect $\Om$. For $1 >a > \frac{1}{2}$, the $h_s$-orbit it intersects our transversal once, on the point $\La_{a, a}$, which is a fixed point for the BCZ map $T$. As we push it down, the $h_s$ trajectory intersects the transversal more and more times. However, it does not simply pick up one intersection at a time, but rather $N(Q+1) -N(Q) = \varphi(Q+1)$ intersections when we transition from $\left(\frac{1}{Q+1}, \frac{Q}{Q+1}\right)$ to $\left(\frac{1}{Q+1}, 1\right)$.

\section{Equidistribution of periodic orbits}\label{subsec:equidist}

\noindent In this section, we prove Theorem~\ref{theorem:equidist:periodic}, which is the key result for our number theoretic applications. This is essentially a direct consequence of the well-known equidistribution principle for closed horocycles on $X_2$, originally due to Sarnak~\cite{Sarnak} and reproved by Eskin-McMullen~\cite{EM} using ergodic theoretic techniques.  In our notation, and in terms of the map $T$, this result can be reformulated as follows. Recall that given a (non-empty) interval  $I = [\alpha, \beta] \subset [0, 1]$, and $Q > 1$, we defined $\rho_{Q, I}$ as the probability measure supported on (a long piece of) the orbit of $(\frac{1}{Q}, 1)$. That is, setting $N_I(Q): = |\F(Q) \cap I|$, we had $$\rho_{Q, I} = \frac{1}{N_I(Q)}\sum_{i: \gamma_i \in I} \delta_{T^i( \frac{1}{Q}, 1)}.$$ Define the measure $\rho^{R}_{Q, I}$ on $X_2$ by setting $$d\rho^R_{Q, I} = d\rho_{Q, I} dt,$$ where we are identifying $X_2$ with the suspension space over $\Om$.

\begin{Theorem}\label{theorem:sarnak}[\cite{Sarnak},\cite{EM},\cite{Hejhal2}] As $Q \rightarrow \infty$, $$\rho^R_{Q, I} \rightarrow \mu_2,$$ where the convergence is taken in the weak-* topology.\end{Theorem}

\noindent\textit{Remark.} In \cite{Sarnak} and \cite{EM} the result is only established for the case $\alpha=0, \beta=1$, although it is clear that the same arguments in \cite{EM} work for the case of fixed $0<\alpha<\beta<1$.  A proof of this was outlined in \cite{Hejhal2}.  Stronger results where the difference $\beta-\alpha$ is permitted to tend to zero with $Q$ have been obtained by Hejhal~\cite{Hejhal3} and Str\"ombergsson~\cite{Strombergsson}.  
\medskip

Let $\pi$ be the projection map from the suspension onto $\Omega$.  Note that $\pi$ is continuous except for a set of measure zero with respect to the measures $\rho^R_{Q,I}$ and $\mu_2$.  Thus, we have as $Q\to\infty$, 
$$ \rho_{Q,I} = \frac{1}{R}\pi_*\rho^R_{Q,I} \to \frac{1}{R}\pi_*\mu_2 = m,$$  
which completes the proof of Theorem 3.1.

\subsection{Farey Statistics}\label{sec:farey} Our results on the statistics of Farey fractions are all immediate corollaries of Theorem~\ref{theorem:equidist:periodic}, being statements of the form \begin{equation}\label{eq:G}\int_{\Om} G d\rho_{Q,I} \rightarrow \int_{\Om} G dm\end{equation} for appropriate choices of functions $G$ on $\Om$. In the applications to spacings and $h$-spacings, the function $G$ is in fact an indicator function, and so the convergence is immediate. In the setting of indices and moments, the functions $G$ are $L^1(m)$ functions, we need the following lemma.  

\begin{lemma}\label{lem:L1}
For $Q\in\N$, set $m_Q=\rho_{Q,I}$ and let $m_\infty=m$.  
Let $W$ be the set of measurable functions $G:\Omega\to\R$ satisfying 
\begin{equation}\label{ieq:upperbound}
  \|G\|_W := \sup_{Q\in\N\cup\{\infty\}} \int_\Omega |G| dm_Q < \infty 
\end{equation}
  where two functions in $W$ are identified if their difference has zero norm.  
Let $C_0(\Omega)$ be the space of continuous functions vanishing at infinity equipped 
  with the uniform norm and let $W_0$ be the closure of its image in $W$.  
Then (\ref{eq:G}) holds for any $G\in W_0$.  
\end{lemma}
\begin{proof}
Let $\eta:C_0(\Omega)\to W$ be the natural inclusion map, which we note is a continuous injection 
  that fails to be an embedding.  It induces a map $\eta^*:W^*\to M(\Omega)$ between the dual $W^*$ 
  with the weak-$^*$ topology and the space $M(\Omega)$ of Radon measures on $\Omega$.  
Each $m_Q$ determines a linear map $\eta(C_0(\Omega))\to\R$ that is readily seen to be 
  continuous even with the subspace topology on $\eta(C_0(\Omega))$ coming from $W$.  
This linear map extends uniquely to a linear map $W_0\to\R$ of norm at most one.  
Given $\eps>0$ and $G\in W_0$, there is sequence $G_k\in C_0(\Omega)$ converging to $G$.  
Choose $k$ so that $\|G_k-G\|<\eps/3$, then choose $Q_0$ such that for all $Q>Q_0$ we have 
  $\left| m_Q(G_k)-m(G_k) \right| < \eps/3$.  The triangle inequality now gives 
  $$| m_Q(G)-m(G) | \le 2\|G-G_k\| + | m_Q(G_k) - m(G_k) | < \eps.  $$
This shows that (\ref{eq:G}) holds.  
\end{proof}

\noindent Below, we indicate which functions go with which corollaries, and show how to verify the condition (\ref{ieq:upperbound}) in these settings. 

\subsubsection{Spacings and $h$-spacings}\label{subsec:spacing} Corollary~\ref{cor:hall} follows from applying (\ref{eq:G}) to the indicator function of the set $R^{-1}\left(\left[\frac{3}{\pi^2 |I|}c,\frac{3}{\pi^2 |I|}d\right]\right)$. Similarly, Corollary~\ref{cor:ABCZ} follows from applying (\ref{eq:G}) to the indicator function of the set $R_h^{-1}(\tilde{B})$. SInce indicator functions are bounded, (\ref{ieq:upperbound}) holds immediately. \qed

\subsubsection{Indices}\label{subsec:index} To verify (\ref{ieq:upperbound}) for $\kappa^{\alpha}$, we need:

\begin{lemma}[\cite{BCZ2}, Lemma 3.4] Let $r \in \N$, $n \geq 4r + 2$. Suppose $\kappa(a,b) = n$. Then
\begin{enumerate}
\item For $i = \pm 1$, $\kappa (T^i(a,b)) = 1$
\item For $1 < |i| \le r$, $\kappa (T^i(a,b)) = 2$
\end{enumerate}
\end{lemma}
\noindent That is, any large value $n$ of $\kappa$ must be followed by roughly $n/4$ small values. Furthermore, by the definition of $\kappa$, we can see that the maximum of $\kappa$ along the support of $m_Q$ is at most $2Q$, which is on the order of the square root of the length of the orbit $N_I(Q)$, since $$\kappa\left(\frac{q_i}{Q}, \frac{q_{i+1}}{Q} \right) = \left \lfloor \frac {1+ \frac{q_i}{Q}}{\frac{q_{i+1}}{Q}} \right \rfloor = \left \lfloor \frac{Q+q_i}{q_{i+1}} \right \rfloor \le \frac{2Q}{1} = 2Q.$$ Combining these two facts, we have that there is a constant $c = c_{\beta}$, depending on the interval $I$ and the power $\beta$ so that for any $0 < \beta < 2$,  and any $Q \in \N$, $$m_Q( \kappa > N) \le c_{\beta} N^{-\beta}.$$ Let $0 < \alpha<\beta < 2$. Then $$m_Q(\kappa^{\alpha}) = \sum_{N=1}^{\infty} m_Q\left(\kappa > N^{\frac {1}{\alpha}}\right) < c_{\beta}\sum_{N=1}^{\infty} N^{-\frac{\beta}{\alpha}} < c_{\beta} \zeta\left(\beta/\alpha\right).$$ Since $m(\kappa = k) = \frac{8}{k(k+1)(k+2)}$ for $k \geq 2$, we have $$m(\kappa^{\alpha}) < 8 \zeta(3-\alpha),$$ so we have verified (\ref{ieq:upperbound}) for $\kappa^{\alpha}$, proving Corollary~\ref{cor:BGZ}.
\qed

\subsubsection{Moments}\label{subsec:moments} Finally, to obtain Theorem~\ref{theorem:classic}, we apply (\ref{eq:G}) to the function $$G_{s,t}(a, b) = a^s b^t,$$ $t, s \in \C$, $\Re s, \Re t \geq  -1$. To verify (\ref{ieq:upperbound}), we use that for $(a, b) \in \Om$, $$|G_{s,t}(a, b)| < |G_{-1,-1}(a,b)|,$$ so it suffices to verify the condition in the case $(s, t) = (-1, -1)$ . We need the following: 

\begin{lemma}\label{lem:counting} 
For any interval $I \subset [0,1)$ there is a constant $c_1$ such that for any $Q$ with $N_I(Q) >0$, 
we have $$N_I(Q) \ge c_1 Q^2.$$ 
\end{lemma}
\begin{proof} $N_I(Q)$ counts the number of primitive lattice points in the triangle: $$\left\{(x,y) \in \R_+^2: \frac y x \in I, x \le Q\right\},$$ which has area $\frac{1}{2}Q^2|I|$.  
There are two cases.  If $N_{Q/2}(I)\ge1$, then Theorem~3 of \cite{Ch2} (with the sup norm) implies 
  $$N_Q(I) \ge \left(\frac{4}{27\pi}\right) \cdot \frac38Q^2|I| = \frac{|I|}{18\pi}Q^2.$$  
On the other hand, if $N_{Q/2}(I)=0$, the length of $I$ is bounded above by the largest 
gap in $F_{Q/2}$, which is at most $2/Q$, so that 
  $$N_Q(I) \ge 1 \ge \frac{Q^2|I|^2}{4}.$$  
Thus, the lemma follows with the choice $$c_1=\min\left(\frac{|I|}{18\pi},\frac{|I|^2}{4}\right).$$  
\end{proof}
\noindent We have $$m_Q(G_{-1,-1}) = \frac{1}{N_I(Q)} \sum_{\gamma_i \in \F_I(Q)} \frac{Q^2}{q_i q_{i+1}} = \frac{|I| Q^2}{N_I(Q)}.$$ Applying Lemma~\ref{lem:counting}, we have that $m_Q(G_{-1,1})$ is unfiormly bounded in $Q$, and we have that $m(G_{-1,1}) = \frac{\pi^2}{3}$, completing the proof.\qed

\subsubsection{Excursions}\label{subsubsec:excursions:proof} To prove Corollary~\ref{cor:farey:amusing}, we apply (\ref{eq:G}) to the functions $$G_{\max}(a,b) = \max\left\{a, b, \frac{1}{a+b}\right\},$$ and $$G_{\min} (a, b) = \min\left\{ \frac 1 a, \frac 1 b, a+b \right\}$$ respectively. Equation (\ref{ieq:upperbound}) is satisfied since $G_{\max}$ is bounded by $1$ and $G_{\min}$ is bounded by $2$.\qed

\section{Piecewise linear description of horocycles }\label{sec:plh}
\noindent Before proving Theorems~\ref{theorem:minima} and~\ref{theorem:maxima}, we record an elementary lemma on the behavior of the length of vectors under $h_s$, whose proof we leave to the reader.

\begin{lemma}\label{lem:PL}
Let $\vv = (x,y)^T\in\R^2$ be a vector with $x>0$.  
Let $\sigma^\pm=\sigma\pm|x|^{-1}$ where $\sigma=\frac{y}{x}$ is the slope of $\vv$.  
Then $f(s)=\|h_s\vv\|_{\sup}$ is the continuous piecewise linear function given by 
\begin{enumerate}
  \item[(i)] $f'(s)=-|x|$ for $s<\sigma^-$, 
  \item[(ii)] $f(s)=|x|$ for $\sigma^-\le s\le \sigma^+$, and 
  \item[(ii)] $f'(s)=+|x|$ for $s>\sigma^+$.  
\end{enumerate}
\end{lemma}

\subsection{Minima and maxima}\label{subsec:minima}
Recall that Lemma~\ref{lem:short} states that $\Om$ consists of lattices with a short (length $\le 1$) horizontal vector. This observation shows that the visits of the trajectory $\{h_s \La\}_{s \geq 0}$ to the transversal $\Om$ are precisely the times $s_n(\La)$ of local minima of $\ell_{\La} (s)$, since a lattice cannot contain more than $1$ vector of supremum norm at most $1$, and under $h_s$, vectors are shorter when they are horizontal. On $\Om$, we have $$\ell\left(\La_{a,b}\right) = a.$$ By abuse of notation we write $\ell: \Om \rightarrow (0, 1]$ with $$\ell(a,b) = a.$$ Similarly we write $\alpha(a,b) = \frac{1}{a}.$ Thus when $h_{s_n} \La = \La_{a_n, b_n} $, the length of the shortest vector in $h_{s_n} \La$ is given by the horizontal vector $\left(\begin{array}{c}a_n \\0\end{array}\right)$. Setting $$s_1 (\La) = \min\{s \geq 0: h_s \La \in \Om\},$$ and writing $\La_{a_1, b_1} = h_{s_1}\La$, we have that 
$$l_N(\La) = \frac{1}{N}\sum_{n=1}^N \ell_{\La}(s_n) = \frac{1}{N} \sum_{n=1}^N \ell\left(T^{n-1}\left(a_1, b_1\right)\right),$$ and
$$a_N(\La): = \frac{1}{N}\sum_{n=1}^N \alpha_{\La}(s_n) =  \frac{1}{N}\sum_{n=1}^N  \alpha\left(T^{n-1}\left(a_1, b_1\right)\right).$$
Since both $\alpha$ and $\ell$ are in $L^1(\Om, dm)$, and we have assumed that $\La$ is not periodic under $h_s$ (and so, $(a_1, b_1)$ is not under $T$), we can apply the Birkhoff Ergodic Theorem (and the fact that all non-periodic orbits are generic for $m$, by the measure classification result) to conclude $$\lim_{N \rightarrow \infty} \frac{1}{N}\sum_{n=1}^N \ell_{\La}(s_n) = \lim_{N \rightarrow \infty} \frac{1}{N}\sum_{n=1}^N  \ell\left(T^{n-1}\left(a_1, b_1\right)\right) = \int_{\Om} \ell dm = \frac{2}{3}. $$ and
$$\lim_{N \rightarrow \infty} \frac{1}{N}\sum_{n=1}^N \alpha_{\La}(s_n)=  \lim_{N \rightarrow \infty} \frac{1}{N}\sum_{n=1}^N  \alpha\left(T^{n-1}\left(a_1, b_1\right)\right) = \int_{\Om} \alpha dm = 2 ,$$ proving Theorem~\ref{theorem:minima}.\qed\medskip


\subsection{Returns}\label{subsec:return:proofs} For the corresponding results for local \emph{maxima} $S_n(\La)$, we need to consider the `hand-off' between the basis vectors $(a, 0)^T$ and $(b, a^{-1})^T$ which happens after hitting the point $\La_{a,b}\in \Om$. This hand-off happens when the vectors $h_s  (a,0)^T= (a,-sa)^T$ and $h_s (b, a^{-1})^T = (b,-sb + a^{-1})^T$ have the same length. That is, we need to calculate the time $s \in (0, R(a,b))$ so that $$\max(a, sa) = \max(b, -sb + a^{-1}),$$ where we have taken absolute values of the coordinates (see Figure~\ref{fig:handoff}). A case-by-case analysis shows that the length of the vector at the hand-off is given by $f(x,y) = \max\left\{a, b, \frac{1}{a+b}\right\}$. By a similar analysis as above, we apply the ergodic theorem to this function (and to its reciprocal) to obtain Theorem~\ref{theorem:maxima}. \qed

\begin{figure}\caption{The hand-off between the vectors \textcolor{blue}{$(a, -sb)^T$} and \textcolor{red}{$(b, -sa+b^{-1})^T$}.}

\label{fig:handoff}

\begin{tikzpicture}[scale=2]

\draw (0, 2) node[left]{\tiny $a^{-1}$}--(0, 3/2) node[left]{\tiny $b^{-1}$}--(0, 1) node[left]{\tiny $1$}--(0, 2/3) node[left]{\tiny $b$}--(0,0)node[below]{\tiny $(0,0)$}--(1,0)node[below]{\tiny $1$}--(2,0)node[below]{\tiny $(a^{-1}b^{-1} -1)$}--(3,0)node[below]{\tiny $a^{-1}b^{-1}$};

\draw[blue] (0, 1/2) node[left, blue]{\tiny $a$}-- (1, 1/2)--(3,3/2)node[right, blue]{\tiny $\max(a, sa)$};

\draw[red] (0, 3/2) node[left, red]{\tiny $$}-- (2, 2/3)--(3, 2/3)node[right, red]{\tiny $\max(b, -sb +a^{-1})$};

\draw[dashed, red] (2, 2/3)--(2, 0);

\draw[dashed, red] (2, 2/3)--(0, 2/3);

\draw[dashed, blue] (1, 1/2)--(1, 0);

 \end{tikzpicture}
 \end{figure}
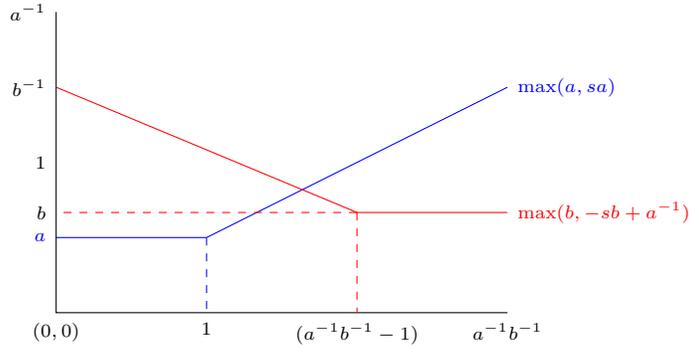

\section{Geometry of Numbers}\label{sec:geomnum} In this section we prove Theorem~\ref{theorem:generic:lattice} and Corollary~\ref{cor:periodic:lattice:gap}. 

\subsection{Slope gap distribution}\label{subsec:slopegap} We will prove Theorem~\ref{theorem:generic:lattice} by noting that the observations of \S\ref{sec:obs:slope} yield the following:

\begin{lemma}\label{lem:slope} Let $t>0$, and let $T_t$ denote the $t$-BCZ map. Let $\La$ denote a lattice without $t$-short vertical vectors, and let $\La_1 = h_{s_1} \La$ be the first intersection of the orbit $\{h_s\La\}_{s >0}$ with the transversal $\Om_t$. Let $0 \le c \le d \le \infty$. Let $\chi_{c, d}$ denote the indicator function of the set $$R^{-1}(c, d) = \left\{ (x,y) \in \Om_t: \frac 1 d < xy < \frac 1 c \right\}.$$ Then $$\frac{1}{N} | G_{N,t}(\La) \cap (c, d)| = \frac 1 N \sum_{i=0}^{N} \chi_{c, d} (T_t^n (\La_1)).$$ \end{lemma}

\begin{proof} The lemma follows from the fact that $h_s$ preserves differences in slopes, and the function $R$ captures these differences, for vectors in the strip $V_t$. In particular, $$R(T_t^n(\La_1)) = s_{n+1} - s_n,$$ which yields the lemma. \end{proof}
\noindent The proof of the theorem then follows as in \S\ref{subsec:minima}, by applying the Birkhoff ergodic theorem and Dani's measure classification. \qed
\subsection{Slope gap statistics}\label{subsec:slopegapstat} To prove Corollary~\ref{cor:periodic:lattice:gap}, we first observe that if $\La$ has a vertical vector, its periodicity under $h_s$ implies that the sequence of slope gaps in $V_t$ must repeat, so in particular there is an $N_0 = N_0(t)$ so that $G_{N,t}(\La) = G_{N_0, t} (\La)$ for all $N \geq N_0$.  For example, in \S\ref{subsec:fareyperiod} we see that for the lattice $\Z^2$, $N_0(t) = N(\lfloor t \rfloor)$, where $N(Q)$ denotes the cardinality of the Farey sequence. Again using the conjugation relation (\ref{eq:conj}), we see that $G_{N, t} (\La) = G_{N, 1}(g_{-2 \log t} \La).$ Using Lemma~\ref{lem:slope}, we see that $\frac{1}{N} | G_{N,t}(\La) \cap (c, d)|$ now corresponds to the integral of $\chi_{c, d}$ along a periodic orbit for $T$, which, as $t \rightarrow \infty$, is getting longer and longer. By Theorem~\ref{theorem:equidist:periodic}, this converges to the integral of $\chi_{c,d}$ with respect to Lebesgue measure $m$, as desired. \qed

%
%
%
%
%
%
%
%
%
%
%
%
%
%
%

\section{Further Questions}\label{sec:further}

\subsection{BCZ maps for other lattices} It was communicated to us by O.~Sarig that it is well-known that the horocycle flow on $SL(2, \R)/\Gamma$ for any lattice $\Gamma$ can be realized as a suspension flow over an adic transformation. A natural conjecture, then is:
\begin{conj} The BCZ map is an adic transformation. \end{conj}
\noindent To our knowledge, the BCZ map is the first \emph{explicit} description of a return map. We have the following immediate
\begin{question*} How explicitly can the return maps for $h_s$ on $SL(2, \R)/\Gamma$ for $\Gamma \neq SL(2, \Z)$ be given?
\end{question*}

The first author, in joint work with J.~Chaika and S.~Lelievre~\cite{ACL}, has calculated a BCZ-type map for $\Gamma = \Delta(2, 5, \infty)$, the $(2, 5, \infty)$-Hecke triangle group (note that $SL(2, \Z)$ is the $(2, 3, \infty)$ triangle group). 

\subsection{Slope gaps for translation surfaces} The motivation for~\cite{ACL} was to generalize the calculation of the gap distribution for Farey sequences, which correspond to closed trajectories for geodesic flow on the torus, to a particular example of a higher-genus translation surface, given by a particular $L$-shaped polygon known as the \emph{golden L}. A more general question would be:
\begin{question*} Is there a $BCZ$-type map for the space of translation surfaces.  Can it be used to explicitly compute gap distributions for saddle connection directions. 
\end{question*}

\subsection{Mixing properties} While the horocycle flow $h_s$ is known to be mixing on $X_2$ with respect to the measure $\mu_2$ (by, e.g., the Howe-Moore theorem~\cite{BekkaMayer}), mixing is a property that does \emph{not} pass between flows and return maps. In fact, there are many well-known constructions of mixing suspension flows over non-mixing base transformations. Thus, we have the natural
\begin{question*} Is the BCZ map $T$ mixing? \end{question*}

\subsection{Rates of convergence} Flaminio-Forni~\cite{FF} and Str\"ombergsson~\cite{Strombergsson} have proved essentially optimal results for the \emph{deviation} of ergodic averages for $h_s$ on $X_2$. Again, precise control of deviations does not pass to the base transformation, so we have the following:
\begin{question*} Can one give bounds on the deviation of ergodic averages for the BCZ map. \end{question*}

\noindent Zagier~\cite{Zagier} showed that proving an optimal rate of equidistribution for long periodic trajectories for $h_s$ on $SL(2, \R)/SL(2, \Z)$ (that is, an optimal error term in Sarnak's theorem~\cite{Sarnak}) is equivalent to the classical Riemann hypothesis. There is also an equivalent formulation of the Riemann hypothesis in terms of the distribution of the Farey sequence due to Franel-Landau~\cite{FL}. This leads to
\begin{question*} Is there an optimum bound on the error term in Theorem~\ref{theorem:equidist:periodic} that is equivalent to the Riemann hypothesis?  
\end{question*}

\subsection{Acknowledgments and Funding}\label{subsec:ack}The original motivation for this project was the study of cusp excursions for horocycle flow on $X_2$ (as described in \S\ref{subsubsec:plh}). We developed the first return map $T$ as a technical tool in order to study these excursions. Later, the first-named author was attending a talk by F.~Boca in which the map $T$ was written down in the context of studying gaps in angles between hyperbolic lattice points. After many invaluable and illuminating conversations with F.~Boca and A.~Zaharescu, we understood the deep connections they had developed between orbits of this map and the study of Farey sequences. It is a pleasure to thank them for their insights and assistance. We would also like to thank Jens Marklof for numerous helpful discussions and insightful comments on an early version of this draft.

We thank the organizers of the Summer School and Conference on Teichm\"uller dynamics in Marseilles in June 2009, when this project was originally developed. In addition, J.S.A. would like to thank San Francisco State University for its hospitality. J.S.A. supported by National Science Foundation grant number DMS-1069153. Y.C. supported by National Science Foundation CAREER grant number DMS-0956209.


\begin{thebibliography}{99}
\bibitem{Arnoux} P.~Arnoux, \textit{Le codage du flot g\'eod\'esique sur la surface modulaire},
L'Enseignement MathŽmatique, Vol.40 (1994)

\bibitem{A} J.~S.~Athreya, \textit{Cusp excursions on parameter spaces}, preprint,  arXiv:1104.2797.

\bibitem{ACL} J.~S.~Athreya, J.~Chaika, and S.~Lelievre, \textit{The distribution of gaps for saddle connections on the golden L}, in preparation.

\bibitem{AM1} J.~S.~Athreya and G.~A.~Margulis, \textit{Logarithm laws for unipotent flows, I}, Journal of Modern Dynamics {\bf 3} (2009), no. 3, 359-378.


\bibitem{ABCZ} V.~Augustin, F.~P.~Boca, C.~Cobeli, and A.~Zaharescu, \textit{The h-spacing distribution between Farey points.} Math. Proc. Cambridge Philos. Soc. {\bf 131} (2001), no. 1, 23 - 38. 

\bibitem{BekkaMayer} M.~B.~Bekka and M.~Mayer, \emph{Ergodic Theory and
Topological Dynamics of Group Action on Homogeneous Spaces}, London
Mathematical Society Lecture Notes v. 269, Cambridge University Press, 2000.

\bibitem{BZsurvey} F.~P.~Boca and A.~Zaharescu, \textit{ Farey fractions and two-dimensional tori}, in \textit{Noncommutative Geometry and Number Theory} (C. Consani, M. Marcolli, eds.), Aspects of Mathematics E37, Vieweg Verlag, Wiesbaden, 2006, pp. 57-77.

\bibitem{BCZ} F.~Boca, C.~Cobeli, and A.~Zaharescu, \textit{A conjecture of R. R. Hall on Farey points.} J. Reine Angew. Math.  {\bf 535} (2001), 207 - 236.

\bibitem{BCZ2} F.~Boca, C.~Cobeli, and A.~Zaharescu, \textit{On the distribution of the Farey sequence with odd denominators.}
Michigan Math. J. {\bf 51} (2003), no. 3, 557 - 573. 

\bibitem{BGZ} F.~Boca, R.~Gologan, and A.~Zaharescu, \textit{On the index of Farey sequences.}
Q. J. Math. {\bf 53} (2002), no. 4, 377 - 391.

\bibitem{BPPZ} F.~Boca, V.~Pasol, A.~Popa, and A.~Zaharescu, \textit{Pair correlation of angles between reciprocal geodesics on the modular surface}, preprint, arXiv:1102.0328.

\bibitem{Bowen} R.~Bowen, \emph{Entropy for group endomorphisms and homogeneous spaces,} 
Trans. AMS {\bf 153} (1971), 401 - 414.

\bibitem{Ch2} Y.~Cheung, \emph{Slowly divergent geodesics in moduli space,} 
     Conf. Geom. Dyn. {\bf 8} (2004), 167--189.  

\bibitem{Cheung} Y.~Cheung, \emph{Hausdorff dimension of the set of singular pairs}, Ann. Math. {\bf 173}, (2011), 127 - 167.  

\bibitem{Dani} S.~G.~Dani, \textit{On uniformly distributed orbits of certain horocycle flows.}
Ergodic Theory Dynamical Systems {\bf 2} (1982), no. 2, 139 - 158 (1983).

\bibitem{EM} A.~Eskin and C.~McMullen, \emph{Mixing, counting, and equidistribution in Lie groups,} Duke Math. J. {\bf 71} (1993), 181 - 209.  

\bibitem{FF} L.~Flaminio and  G.~Forni, \emph{Invariant distributions and time averages for
horocycle flows}, Duke Math. J., v. 119, 465-526, 2003.

\bibitem{FL} J.~Franel and E.~Landau, \emph{Les suites de Farey et le problme des nombres premiers,} G\"ottinger Nachr., 198Ð206

\bibitem{Gurevic} B.~M.~Gur\v{e}vic, \emph{The entropy of a horocycle flow}, Soviet Math. Doklady {\bf 2} (1961), 124 - 130.

\bibitem{Hall} R.~R.~Hall, \textit{A note on Farey series.}
J. London Math. Soc. (2) {\bf 2} 1970 139 - 148.

\bibitem{Hall-Shiu} R.~R.~Hall and P.~Shiu, \textit{The index of a Farey sequence.}
Michigan Math. J. {\bf 51} (2003), no. 1, 209Ð223. 

\bibitem{HK} M.~Handel and B.~Kitchens, \emph{Metrics and entropy for non-compact spaces. With an appendix by Daniel J. Rudolph.} Israel J. Math. {\bf 91} (1995), no. 1-3, 253 - 271.

\bibitem{HT} R.~R.~Hall and G.~Tenenbaum, \textit{On consecutive Farey arcs.}
Acta Arith. {\bf 44} (1984), no. 4, 397 - 405.

\bibitem{Hedlund} G.~Hedlund, \textit{Fuchsian groups and transitive horocycles},
Duke Math. J. {\bf 2} (1936), no. 3, 530 - 542. 

\bibitem{Hejhal2} D.A.~Hejhal, \emph{On value distribution properties of automorphic
functions along closed horocycles} in \textit{XVI}th Rolf Nevanlinna Colloquium
(Joensuu, Finland, 1995), de Gruyter, Berlin, 1996, 39 - 52. 

\bibitem{Hejhal3} D.A.~Hejhal, \emph{On the uniform equidistribution of long closed
horocycles} in Loo-Keng Hua: A Great Mathematician of the Twentieth Century, 
Asian J. Math. {\bf 4}, Int. Press, Somerville, Mass., 2000, 839 Ð 853.  

\bibitem{KZ} %
P.P.~Kargaev and A. A.~Zhigljavsky, \emph{Asymptotic distribution of the distance 
function to the Farey points}, J. Number Theory {\bf 65} (1997) 130-149.  

\bibitem{Marklof1} 
J.~Marklof, \emph{The asymptotic distribution of Frobenius numbers}, 
Invent. Math. {\bf 181} (2010) 179 - 207.

\bibitem{Marklof2} 
J.~Marklof, \emph{Fine-scale statistics for the multidimensional Farey sequence},
   preprint: arXiv:1207.0954

\bibitem{MS} 
J.~Marklof and A.~Stršmbergsson, \emph{The distribution of free path lengths in 
the periodic Lorentz gas and related lattice point problems}, 
Ann. of Math.  {\bf 172} (2010) 1949 - 2033.  

\bibitem{Nadkarni} M.~G.~Nadkarni, \textit{Basic ergodic theory.}
Second edition. BirkhŠuser Advanced Texts, BirkhŠuser Verlag, Basel, 1998.

\bibitem{Sarnak} P.~Sarnak, \textit{Asymptotic behavior of periodic orbits of the horocycle flow and Eisenstein series.} Comm. Pure Appl. Math. {\bf 34} (1981), no. 6, 719 - 739.

\bibitem{Series} C.~Series, \textit{The Modular Surface and Continued Fractions}, J. London Math. Soc. (2), {\bf 31} (1985), 69-80.

\bibitem{Strombergsson} A.~Str\"ombergsson, \emph{On the uniform equidistribution of long closed horocycles,} Duke Math. J. {\bf 123} (2004), 507 - 547.  

\bibitem{Walters} P.~Walters, \textit{An introduction to ergodic theory}, in Graduate Texts in Mathematics 79 (Springer-Verlag, New York, 1982).

\bibitem{Zagier} D.~Zagier, \textit{Eisenstein series and the Riemann zeta function} in \textit{Automorphic Forms,
Representation Theory and Arithmetic} (Bombay, 1979), Tata Inst. Fund. Res. Studies in Math. 10, Tata Inst. Fund. Res., Bombay, 1981, 275 Ð 301. MR 83j:10027

\end{thebibliography}
\end{document}